\documentclass[11pt]{article}

\newcommand{\hide}[1]{}

\usepackage[maxbibnames=99]{biblatex}
\usepackage{fancybox,graphicx}
\usepackage[title]{appendix}
\usepackage{xcolor}
\usepackage{url}
\usepackage[linkbordercolor={1 0 0}]{hyperref}
\usepackage{mathtools}
\usepackage{amssymb}
\usepackage{algorithm}
\usepackage[noend]{algpseudocode}
\usepackage{authblk}

\hide{
\usepackage{amscd}
\usepackage{amsfonts}
\usepackage{amsmath}
\usepackage{amssymb}
\usepackage{amsthm}
\usepackage{cases}		
\usepackage{cutwin}
\usepackage{enumerate}
\usepackage{epstopdf}
\usepackage{graphicx}
\usepackage{ifthen}
\usepackage{lipsum}
\usepackage{mathrsfs}	
\usepackage{multimedia}
\usepackage{wrapfig}
}
\usepackage{amsthm}
\newtheorem{theorem}{Theorem}[section]

\newlength\tindent
\newcommand{\VI}{\mbox{VI}}
\newcommand{\RR}{\mathbb R}
\newcommand{\XX}{{\cal X}}
\newcommand{\T}{{\footnotesize\mbox{th}}}
\newcommand{\half}{0.5}

\title{ An Approximation-Based Regularized Extra-Gradient Method for Monotone Variational Inequalities }

\author{
Kevin Huang\thanks{Department of Industrial and System Engineering, University of Minnesota, huan1741@umn.edu}
\hspace{1cm}
Shuzhong Zhang\thanks{Department of Industrial and System Engineering, University of Minnesota, zhangs@umn.edu}
}

\date{\today}

\begin{document}

\maketitle

\begin{abstract}
    In this paper, we propose a general extra-gradient scheme for solving monotone variational inequalities (VI), referred to here as {\it Approximation-based Regularized Extra-gradient method} (ARE). The first step of ARE solves a VI subproblem with an approximation operator satisfying a $p^\T$-order Lipschitz bound with respect to the original mapping, further coupled with the gradient of a $(p+1)^{\T}$-order regularization. 
    %with a $(p-1)^{\T}$-order approximation of the original mapping at the current iterate, coupled with the gradient of a $(p+1)^{\T}$-order regularization. 
    The optimal global convergence is guaranteed by including an additional extra-gradient step, while a $p^{\T}$-order superlinear local convergence is 
    %even possible when 
    shown to hold if the VI is strongly monotone. The proposed ARE is inclusive and general, in the sense that %seemingly 
    a variety of solution methods can be formulated within this framework as different manifestations of approximations, and their iteration complexities would follow through 
    %is analyzed 
    in a unified fashion. The ARE framework relates to the first-order methods, while opening up possibilities to developing higher-order methods specifically for structured problems that guarantee the optimal iteration complexity bounds.
    
    %This not only generalizes many common first-order methods when $p=1$, but also significantly simplifies both the procedure and analysis for higher-order methods when $p>1$.
     
    %ARE solves a VI subproblem associated with a regularized approximation operator in the first step, followed by a normal extra-gradient update in the second step. The concept of approximation enables ARE to include a much broader range of specific methods under the general extra-gradient scheme, and the proper regularization ensures optimal convergence, in the sense of matching either the lower bound or the best in the literature. We also discuss specific methods under this general scheme in the context of composite VI with operator $F(x)=H(G(x))$ and demonstrate applications that either suggest potentially new methods or are connected to existing ones.
    \vspace{3mm}

\end{abstract}

\section{Introduction}
\label{sec:intro}
Let $\XX \subset \RR^n$ be a convex set; $F(x): \RR^n \mapsto \RR^n$ be a vector mapping. The following problem is known as the {\it variational inequality problem}\/ (VI):
\[
\mbox{ Find $x^* \in \XX$ such that $F(x^*)^\top (x-x^*) \ge 0$ for all $x\in \XX$}.
\]
As a notation we denote the solution set as
\[
\VI_{\XX}(F(x)):=\{ x^* \mid \mbox{$x^* \in \XX$ such that $F(x^*)^\top (x-x^*) \ge 0$ for all $x\in \XX$} \}.
\]
We assume $\VI_{\XX}(F(x))$ is non-empty throughout this paper. The study of finite-dimensional VI problems dates back to 1960's where the complementarity problem was developed to solve for various equilibria, such as economic equilibrium, traffic equilibrium, and in general Nash equilibrium. For a comprehensive study of the applications, theories and algorithms of VI, readers are referred to the celebrated monograph by Facchinei and Pang \cite{facchinei2007finite}.

In this paper, we are interested in a specific class of VI, where the operator $F$ is {\it monotone}:
\begin{eqnarray}
\langle F(x)-F(y),x-y\rangle\ge\mu\|x-y\|^2,\quad\forall x,y\in\mathcal{X}\label{strong-monotone}
\end{eqnarray}
for some $\mu\ge0$. If there exists some $\mu>0$ such that \eqref{strong-monotone} holds, it is referred to {\it strongly monotone} and $\VI_{\XX}(F(x))$ is a singleton. The earliest methods developed to solve VI of this type are the projection method due to Sibony \cite{sibony1970methodes}:
\begin{eqnarray}
x^{k+1}:=\arg\min\limits_{x\in\mathcal{X}}\quad\langle F(x^k),x-x^k\rangle+\frac{\gamma_k}{2}\|x-x^k\|^2,\label{projection-method}
\end{eqnarray}
and the proximal point method due to Martinet \cite{martinet1970breve}:
\begin{eqnarray}
%x^{k+1}:=\arg\min\limits_{x\in\mathcal{X}}\quad\langle F(x^{k+1}),x-x^k\rangle+\frac{\gamma_k}{2}\|x-x^k\|^2,\label{proximal-point}
x^{k+1}\in \VI_{\XX}(F(x)+\gamma_k(x-x^k)),\label{proximal-point-update}
\end{eqnarray}
for positive $\{\gamma_k\}_{k\ge0}$. 
%The proximal point method is also studied and popularized by Rochafellar \cite{rockafellar1976monotone}. 
These two methods form the basis of most, if not all, methods developed for monotone VI in the research community thus far.
%that follows. 

%The extra-gradient method due to 
Korpelevich \cite{korpelevich1976extragradient} first introduced an {\it extra-step}\/ in the update as follows:
\begin{eqnarray}
\left\{
\begin{array}{lcl}
     x^{k+0.5}&:=&\arg\min\limits_{x\in\mathcal{X}}\quad\langle F(x^k),x-x^k\rangle+\frac{\gamma_k}{2}\|x-x^k\|^2, \\
     &&\\
     x^{k+1}&:=&\arg\min\limits_{x\in\mathcal{X}}\quad\langle F(x^{k+0.5}),x-x^k\rangle+\frac{\gamma_k}{2}\|x-x^k\|^2.
\end{array}
\right.\label{extra-gradient-update}
\end{eqnarray}
The iteration complexity of the extra-gradient method \eqref{extra-gradient-update} is later established by Tseng \cite{tseng1995linear}. In particular, if the operator is strongly monotone ($\mu>0$), it is $\mathcal{O}\left(\kappa\ln\left(\frac{1}{\epsilon}\right)\right)$ for an $\epsilon$-solution, where $\kappa=\frac{L}{\mu}$ is the condition number for Lipschitz continuous operator with constant $L$. This is a significant improvement over $\mathcal{O}\left(\kappa^2\ln\left(\frac{1}{\epsilon}\right)\right)$ of the vanilla projection method \eqref{projection-method}, and it is in fact optimal among first-order methods (i.e.\ using only the information of $F(\cdot)$) applied to such class of problems (with lower bound recently established by Zhang {\it et al.} \cite{junyu2019}). Many algorithms developed for monotone VI thereafter adopt this concept of extra-step update and can be considered as variants of the extra-gradient method, such as modified forward-backward method \cite{tseng2000modified}, mirror-prox method \cite{nemirovski2004prox}, dual-extrapolation method \cite{nesterov2007dual, nesterov2006solving}, hybrid proximal extra-gradient method \cite{monteiro2010complexity}, extra-point method \cite{huang2021unifying}.

To facilitate the discussion, let us first introduce a few terminologies that will be used throughout the paper. The term ``$p^{\T}$-order method'' will be used following the convention of optimization. In particular, by considering $F(x)=\nabla f(x)$ specifically as a gradient mapping of some function $f(x)$, the first-order method in VI refers to using only the information from the operator $F(\cdot)$, and the $p^{\T}$-order method refers to using the $(p-1)^{th}$-order derivative of the operator: $\nabla^{p-1}F$. As a result, the term ``gradient'' will also be used to refer to $F(\cdot)$ due to the background %the deep connection 
of VI in solving saddle-point and optimization models. In the $p^{\T}$-order method, the Lipschitz continuity of $\nabla^{p-1}F(x)$ is assumed with constant $L_p$:
\begin{eqnarray}
\|\nabla^{p-1}F(x)-\nabla^{p-1}F(y)\|\le L_p\|x-y\|.\label{p-lipschitz}
\end{eqnarray}

In this paper, the proposed {\it Approximation-based Regularized Extra-gradient method} (ARE) can be viewed as a generalization of the extra-gradient method \eqref{extra-gradient-update} %in several aspects. 
in some sense. 
While the intermediate iterate $x^{k+0.5}$ in extra-gradient method \eqref{extra-gradient-update} is updated by a gradient projection step, in ARE it is replaced by solving a VI subproblem:
\begin{eqnarray}
x^{k+0.5}:=\VI_{\XX}(\tilde{F}(x;x^k)+\gamma\|x-x^k\|^{p-1}(x-x^k)),\label{p+1-regularize-VI-sub}
\end{eqnarray}
where $\tilde{F}(x;x^k)$ is an approximation mapping at $x^k$ that satisfies a $p^\T$-order Lipschitz bound with respect to $F(x)$ (will be formally defined later), and $\|x-x^k\|^{p-1}(x-x^k)$ is the gradient mapping of a $(p+1)^{\T}$-order regularization. Therefore, we refer to the update in \ref{p+1-regularize-VI-sub} as {\it $(p+1)^\T$-order regularzied VI subproblem}. A common choice of $\tilde{F}(x;x^k)$ is the {\it Taylor approximation} of $F(x)$ at $x^k$, namely,
\[
\tilde{F}(x;x^k):=\sum\limits_{i=0}^{p-1}\frac{1}{i!}\nabla^{i}F(x^k)[x-x^k]^i.
\]
Such choice of $\tilde{F}(x;x^k)$ not only recovers the
extra-gradient method when $p=1$, but also gives a succinct update principle for higher-order methods when $p>1$. However, the Taylor approximation needs not be the only motivation for the %option in 
ARE. We show that the key underlying condition is the aforementioned $p^\T$-order Lipschitz bound, therefore any approximation satisfying this condition can be considered as a valid method under the general framework of ARE. This not only generalizes the existing methods but also opens up the possibilities of developing different methods from those in the literature, and we will discuss several such specific schemes in Section \eqref{sec:structured-are}.
%For example, while the modified forward-backward method \cite{tseng2000modified} for solving the monotone inclusion problem with an operator of summation form $F(x):=H(x)+G(x)$ can also be viewed as a specific instance of ARE, one can consider a more general composite operator $F(x):=H(G(x))$ and develop second-order methods that linearize certain component of such operator. 
By applying the abstraction of ``approximation'' in ARE, a unifying and concise analysis is available to establish the iteration complexity bound that can be readily specified to any concrete approximation in various methods given the different problem structures at hand.

The rest of the paper is organized as follows. Section \ref{sec:literature} reviews relevant first-order and higher-order methods for solving monotone VI. Section \ref{sec:ARE-global} formally presents ARE and analyzes the global convergence for both monotone and strongly monotone cases. Section \ref{sec:ARE-local} continues the discussion with strongly monotone VI and establishes the local superlinear convergence. A modified ARE algorithm is in place to guarantee both global linear and local superlinear convergence. Section~\ref{sec:subproblem} is devoted to the discussion of solving the VI subproblem with the approximation operator under special cases. In Section~\ref{sec:structured-are}, we present several structured ARE schemes given the original operator is of the composite form $F(x)=H(x)+G(x)$ and discuss their connections to the existing methods. We further discuss two specialized approximation concepts: the outer approximation and the inner approximation, given the general composite form $F(x)=H(G(x))$.  Numerical results from preliminary experiments are demonstrated in Section~\ref{sec:numerical}, and we conclude the paper in Section~\ref{sec:conclusion}.

\section{Literature Review}
\label{sec:literature}
%In addition to 
Historically, 
Martinet \cite{martinet1970breve} first introduced the notion of proximal point method, which was later studied and popularized by Rochafellar \cite{rockafellar1976monotone}.  %also studied and popularized the proximal point method. 
Tseng in \cite{tseng1995linear} studied the linear convergence of proximal point method, extra-gradient method, and matrix-splitting method, given a certain error bound is satisfied. The strongly monotone operator can be an immediate example of such case. 
As a matter of fact, subsequently developed methods such as modified forward-backward method \cite{tseng2000modified}, mirror-prox method \cite{nemirovski2004prox}, dual-extrapolation method \cite{nesterov2007dual, nesterov2006solving}, hybrid proximal extra-gradient (HPE) method \cite{monteiro2010complexity}, extra-point method \cite{huang2021unifying} 
%are the methods that based on 
are all based on the concept of extra-gradient.

Another type of method, known as {\it optimistic gradient descent ascent method} (OGDA), was first proposed by Popov \cite{popov1980modification}:
\begin{eqnarray}
x^{k+1}:=P_{\mathcal{X}}\left(x^k-\alpha F(x^k)-\eta(F(x^k)-F(x^{k-1}))\right),\label{ogda-update}
\end{eqnarray}
for some positive $\alpha,\eta>0$, where $P_{\mathcal{X}}$ denotes the projection operator onto $\mathcal{X}$. Unlike the update in extra-gradient method \eqref{extra-gradient-update} which uses an extra step, OGDA only requires one update (one projection) per iteration and uses the information from the previous iterate $x^{k-1}$ instead. The optimal convergence of OGDA, in both monotone and strongly monotone VI, is established by Mokhtari {\it et al}.\ \cite{mokhtari2019unified, mokhtari2020convergence}. {The extra-point method proposed by Huang and Zhang \cite{huang2021unifying} extends %takes in 
the concepts of the extra-gradient method, OGDA, Nesterov's acceleration in optimization \cite{nesterov1983method}, and the ``heavy-ball'' method by Polyak \cite{polyak1964some} and combines them in a unifying update scheme.} If the parameters associated to these different components satisfy a certain constraint set, it is shown that optimal iteration complexity is guaranteed. There is another line of work that studies variants of extra-gradient type methods \cite{yoon2021accelerated, lee2021fast, kim2021accelerated} and proximal point methods \cite{tran2021halpern, lieder2021convergence, park2022exact} with the {\it anchoring} update, where in each iteration the initial iterate is used as the component of convex combination. The iterates produced are shown to converge among these different methods \cite{yoon2022acceleratedFlock}, at a rate same as the optimal convergence rate (to the solution), and the iteration complexities are improved by constant orders compared to vanilla extra-gradient method.

The above methods are known as the {\it first-order}\/ methods. The lower bound of 
the iteration complexity for the 
first-order methods applied to monotone VI is $\Omega\left(\frac{1}{\epsilon}\right)$, as established by Nemirovsky and Yudin \cite{nemirovsky1983problem}, while for strongly monotone VI, it %the lower bound 
is $\Omega\left(\kappa\ln\left(\frac{1}{\epsilon}\right)\right)$, shown by Zhang {\it et al}.\ \cite{junyu2019} in the context of strongly-convex-strongly-concave saddle-point problems. Methods such as extra-gradient method, mirror-prox method, dual-extrapolation method \cite{nesterov2007dual, nesterov2006solving}, HPE, OGDA, extra-point method have been proven to achieve these lower bounds, hence optimal.

The work of Taji {\it et al.} \cite{taji1993globally} is among the first to consider second-order methods for solving VI. A linearized VI subproblem with operator $F(x^k)+\nabla F(x^k)(x-x^k)$ is solved in each iteration and the merit function $f(x)=\max\limits_{x'\in\mathcal{X}}\langle F(x),x-x'\rangle-\frac{\mu}{2}\|x-x'\|^2$ is used to prove the global convergence, with an additional local quadratic convergence. However, no explicit iteration complexity is established for second-order methods until recently. Following the line of research in \cite{taji1993globally}, Huang and Zhang \cite{huang2022cubic} specifically consider unconstrained strongly-convex-strongly-concave saddle point problem and incorporate the idea of cubic regularization (originally proposed by Nesterov in the context of optimization \cite{nesterov2006cubic}), proving the global iteration complexity $\mathcal{O}\left(\left(\kappa^2+\frac{\kappa L_2}{\mu}\right)\ln\left(\frac{1}{\epsilon}\right)\right)$, where $L_2$ is the Lipschitz constant of the Hessian information, in addition to the local quadratic convergence. 

Another line of research on second-order methods was started by Monteiro and Svaiter \cite{monteiro2012iteration}. They propose a Newton Proximal Extragradient (NPE) method, which can be viewed as a special case of the HPE with large step size. In HPE, the first step solves approximately the proximal point update \eqref{proximal-point-update} (denote as $x^{k+0.5}$), while the second step is a regular extra-gradient step. The ``large step size'' condition, which is key to guarantee a superior convergence rate, requires:
\begin{eqnarray}
\frac{1}{\gamma_k}\ge\frac{\theta}{\|x^{k+0.5}-x^k\|}\label{large-step}
\end{eqnarray}
for some constant $\theta>0$. Note that since $x^{k+0.5}$ depends on $\gamma_k$, a certain procedure is required to determine $x^{k+0.5}$ and $\gamma_k$ such that \eqref{large-step} also holds. By observing that the set of $\gamma_k$ satisfying the condition is in fact a closed interval, they develop a bisection method to iteratively reduce the range of $\gamma_k$ and solve for $x^{k+0.5}$ for each fixed $\gamma_k$ until the condition is satisfied. They show that for monotone VI, NPE admits $\mathcal{O}\left(1/\epsilon^{\frac{2}{3}}\right)$ iteration complexity for {\it ergodic mean} of $x^{k+0.5}$ over $0\le k\le N-1$, which is an improvement over the optimal first-order complexity $\mathcal{O}\left(\frac{1}{\epsilon}\right)$. While NPE can also be expressed in the form of second-order mirror-prox method,
% \begin{eqnarray}
% \left\{
% \begin{array}{lcl}
%      x^{k+0.5}&:=&\arg\min\limits_{x\in\mathcal{X}}\quad\langle F(x^{k})+\nabla F(x^k)(x-x^k),x-x^k\rangle+\frac{\gamma_k}{2}\|x-x^k\|^2, \\
%      &&\\
%      \mbox{such that}&& \frac{L_2}{2}\|x^{k+0.5}-x^k\|\le\gamma_k\le L_2\|x^{k+0.5}-x^k\|,\\
%      &&\\
%      x^{k+1}&:=&\arg\min\limits_{x\in\mathcal{X}}\quad\langle F(x^{k+0.5}),x-x^k\rangle+\frac{\gamma_k}{2}\|x-x^k\|^2,
% \end{array}
% \right.\label{NPE-update}
% \end{eqnarray}
Bullins and Lai \cite{bullins2020higher} propose a ``higher-order mirror-prox method'', extending the second-order mirror-prox method to $p^\T$-order and establish $\mathcal{O}\left(1/\epsilon^{\frac{2}{p+1}}\right)$ iteration complexity. They replace the linearization $F(x^k)+\nabla F(x^k)(x-x^k)$ with the Taylor approximation of $F(x^k)$, $\sum\limits_{i=0}^{p-1}\frac{1}{i!}\nabla^{i}F(x^k)[x-x^k]^i$, together with an higher-order constraint on $\gamma_k$ and $x^{k+0.5}$ similar to \eqref{large-step}.
% \begin{eqnarray}
% \frac{16L_p\|x^{k+0.5}-x^k\|^{p-1}}{p!}\le\gamma_k\le\frac{32L_p\|x^{k+0.5}-x^k\|^{p-1}}{p!}.\label{gamma-implicit-constraint}
% \end{eqnarray}
They also demonstrate an explicit procedure to instantiate the proposed method in unconstrained problem with $p=2$ and a bisection method to search for $x^{k+0.5}$ and $\gamma_k$. In \cite{ostroukhov2020tensor}, Ostroukhov {\it et al}.\ further extend the higher-order mirror-prox method to strongly monotone VI by incorporating the {\it restart procedure}, which yields global iteration complexity $\mathcal{O}\left(\left(\frac{L_p}{\mu}\right)^{\frac{2}{p+1}}\ln\left(\frac{1}{\epsilon}\right)\right)$. The local quadratic convergence is then guaranteed by incorporating CRN-SPP proposed in \cite{huang2022cubic}. Nesterov in \cite{nesterov2006cubicVI} proposes solving constrained convex optimization with cubic regularized Newton method and extends the results to monotone VI with cubic regularized Newton modification of the dual-extrapolation method \cite{nesterov2007dual}. The global iteration complexity is shown to be $\mathcal{O}(\frac{1}{\epsilon})$ for monotone VI, with local quadratic convergence for strongly monotone VI.

Recently, there are new developments of higher-order methods for VI that are closely related to the work in this paper. Jiang and Mokhtari \cite{jiang2022generalized} propose the Generalized Optimistic Method, which is a general $p^\T$-order variant of OGDA. Instead of using $F(x^k)$ to approximate the proximal point update direction $F(x^{k+1})$ with correction $F(x^k)-F(x^{k-1})$ as in OGDA \eqref{ogda-update}, they propose to use a general approximation $P(x^{k+1};\mathcal{I}_{k})$ with correction $F(x^k)-P(x^k;\mathcal{I}_{k-1})$, where $P(x;\mathcal{I}_k)$ can contain $p^\T$-order information and $\mathcal{I}_k$ is the information up to $k^\T$ iteration. Unlike ARE, the Generalized Optimistic Method does not require an additional projection step in the update, nor does it require restart to establish linear convergence for strongly monotone VI (see \cite{ostroukhov2020tensor} and the discussion in Section \ref{sec:ARE-global-strong}). However, it still requires to incorporate a bisection subroutine to solve the higher-order subproblem, similar to the higher-order mirror-prox method \cite{bullins2020higher}, while ARE does not. Adil {\it et al}.\ \cite{adil2022optimal} propose a $p^\T$-order method that improves upon the higher-order mirror-prox method in \cite{bullins2020higher}. The improvement comes from incorporating the gradient of $(p+1)^\T$-order regularization in the higher-order VI subproblem, which makes the bisection subroutine unnecessary, and the global complexity is improved by a logarithmic factor. The special case of ARE, where the Taylor approximation is used as the approximation operator $\tilde{F}(x;x^k)$ in the $(p+1)^\T$-order regularized VI subproblem \eqref{p+1-regularize-VI-sub}, coincides with the method proposed in \cite{adil2022optimal} for solving monotone VI. In this paper, we further develop the global and local convergence for strongly monotone VI (see Section \ref{sec:ARE-global-strong} and \ref{sec:ARE-local}) and discuss the possibilities beyond Taylor approximation in the subproblem, which is one of the key underlying motivation behind the ARE framework. Lin and Jordan \cite{lin2022perseus} propose a $p^\T$-order generalization of Nesterov's dual extrapolation method \cite{nesterov2007dual}, referred to Perseus. Same as ARE and \cite{adil2022optimal}, Perseus does not require bisection subroutines by solving the VI subproblem with higher-order regularization. In addition to developing the iteration complexity guarantee in monotone and strongly monotone VI, \cite{lin2022perseus} also extends the analysis to non-monotone VI that satisfies the (strong) Minty condition. Furthermore, they establish the lower bound complexity for general $p^\T$-order method applied to monotone VI, given by $\Omega\left(1/\epsilon^{\frac{2}{p+1}}\right)$, which is achieved by Perseus, the Generalized Optimistic Method \cite{jiang2022generalized}, \cite{adil2022optimal}, and ARE in this paper. Therefore, they are all optimal $p^\T$-order methods for monotone VI.

The above mentioned higher-order methods share some common aspects with ARE, among which the underlying ideas of ARE are more closely related to NPE \cite{monteiro2012iteration} and higher-order mirror-prox method \cite{bullins2020higher, adil2022optimal}, as we shall formally present in Section \ref{sec:ARE-global}. However, the contributions of this paper are also distinguished from the previous work in the following perspectives. Firstly, by identifying the key condition required in developing the higher-order methods, we are able to replace the commonly used Taylor approximation with a more general approximation in the subproblem. This not only includes existing methods under the framework of ARE, but also leads to developing specific ARE schemes for VI with special structures. We devote Section \ref{sec:structured-are} to a more in-depth discussion on this perspective.
%that are likely unseen before, and the advantages are particularly conspicuous for first- and second-order methods, where the approximation often admits to explicit forms and the subproblem is relatively solvable. 
Secondly, we also identify that by solving a $(p+1)^\T$-order regularized VI subproblem as the first step in each iteration, the procedure of the algorithm as well as the corresponding analysis are largely simplified compared to the previous work in \cite{monteiro2012iteration, bullins2020higher}. We also discuss in more details the procedure for solving such regularized VI subproblem in the practical case where $p=2$ in Section \ref{sec:subproblem}. Finally, a $p^\T$-order local superlinear convergence is established for $p^\T$-order ARE for strongly monotone VI, which is new to the literature compared to the existing $\left(\frac{p+1}{2}\right)^\T$-order superlinear convergence established in \cite{jiang2022generalized, lin2022perseus}.

\section{The Global Convergence Analysis of ARE}
\label{sec:ARE-global}
The Approximation-based Regularized Extra-gradient method (ARE) aims to solve the VI problem:
\begin{eqnarray}
\VI_{\XX}(F(x)):=\{ x^* \mid \mbox{$x^* \in \XX$ such that $F(x^*)^\top (x-x^*) \ge 0$ for all $x\in \XX$} \}.\label{VI-prob}
\end{eqnarray}
We assume that $F(x)$ is monotone \eqref{strong-monotone} and $\VI_{\XX}(F(x))$ is non-empty. When $F(x)$ is assumed to be strongly monotone, $\VI_{\XX}(F(x))$ becomes a singleton. We also assume the $p^\T$-order Lipschitz continuity \eqref{p-lipschitz}.

Now, given an arbitrary $y\in\mathcal{X}$, we are interested in a general approximation mapping at $y$: $\tilde{F}(\cdot;y):\RR^n\mapsto\RR^n$, such that the following $p^\T$-order Lipschitz bound holds between the original mapping $F(x)$ and the approximation $\tilde F(x;y)$:
\begin{eqnarray}
\|\Tilde{F}(x;y)-F(x)\|\le \tau L_p\|x-y\|^p,\label{approx-oper-bound}
\end{eqnarray}
for some $p>1$ and $\tau\in(0,1]$. The examples of such approximation include but not limited to the general Taylor approximation, which further includes $\Tilde{F}(x;y)=F(y)$ for $p=1$ and $\Tilde{F}(x;y)=F(y)+\nabla F(y)(x-y)$ for $p=2$ as special cases. In general, we say the proposed ARE is a ``$p^\T$-order method'' if the Lipschitz bound \eqref{approx-oper-bound} holds with $p$.
%We shall refer to the mapping satisfying \eqref{approx-oper-bound} as ``$(p-1)^{\T}$-order approximation of $F(x)$ at $y$'', following the convention of Taylor approximation. 

Based on the approximation mapping $\tilde{F}(x;y)$, let us consider the {\it regularized} approximation mapping by adding a gradient mapping of the $(p+1)^\T$-order regularization term, expressed in the following form:
\begin{eqnarray}
\tilde{F}(x;y)+L_p\|x-y\|^{p-1}(x-y).\label{regularized-approx-map}
\end{eqnarray}
Since the Jacobian of $L_p\|x-y\|^{p-1}(x-y)$ is positive definite for $x\neq y$, the mapping is monotone \cite{facchinei2007finite}. Therefore, the regularized approximation mapping \eqref{regularized-approx-map} is also monotone as long as $\tilde{F}(x;y)$ is. In ARE, the first step in each iteration is solving a VI subproblem with operator \eqref{regularized-approx-map}, i.e.\ a $(p+1)^\T$-order regualarized VI subproblem, followed by an extra-gradient step, summarized as follows:
\begin{eqnarray}
\left\{
\begin{array}{ccl}
x^{k+0.5} &:=& \VI_{\XX} \left(  \tilde F(x;x^k) + L_p \|x-x^k\|^{p-1} (x-x^k) \right) , \\
x^{k+1} &:=& \mbox{arg}\min\limits_{x \in \XX} \, \langle F(x^{k+0.5}) , x - x^k \rangle + \frac{L_p\|x^{k+0.5}-x^k\|^{p-1}}{2} \|x-x^k\|^2,
\end{array}
\right.\label{are-update}
\end{eqnarray}
for $k=1,2,...$. {In the above, as a matter of notion we indicate $x^{k+0.5}$ to be any solution taken from the solution set $\VI_\XX\left(\cdot\right)$.}
In the second step, the extra-gradient step involves a varying step size $\frac{1}{\gamma_k}$, where
\[
\gamma_k:=L_p\|x^{k+0.5}-x^k\|^{p-1},\quad k\ge1
\]
is a parameter depending on the previous update $x^{k+0.5}$. This together with the bound \eqref{approx-oper-bound} form the basis of the optimal iteration complexity bound for ARE.

\subsection{Solving monotone VI with ARE}

We first establish the global convergence results for solving a general monotone VI \eqref{VI-prob} with $p^\T$-order ARE in the next theorem.
\bigskip
\begin{theorem}[Global convergence of ARE: Monotone VI]
\label{thm:are-global}
Let $\{x^k\}_{k\ge1}$ and $\{x^{k+0.5}\}_{k\ge1}$ be generated by \eqref{are-update} and suppose $F(\cdot)$ is monotone and $\tilde{F}(x;x^k)$ is such that \eqref{approx-oper-bound} holds. Then
\begin{eqnarray}
m(\Bar{x}_N):=\max\limits_{x\in\mathcal{X}}\,\langle F(x),\Bar{x}_N-x\rangle\le\frac{D^2}{2\Gamma_N}=\mathcal{O}(N^{-\frac{p+1}{2}}),\label{conv-monotone}
\end{eqnarray}
where
\begin{eqnarray}
\Bar{x}_N:=\frac{\sum\limits_{k=1}^N\frac{x^{k+0.5}}{\gamma_k}}{\Gamma_N},\quad \Gamma_N:=\sum\limits_{k=1}^N\gamma_k^{-1}\nonumber
\end{eqnarray}
for some $N>0$, and $D:=\max\limits_{x,x'\in\mathcal{X}}\|x-x'\|$.
\end{theorem}

\begin{proof}
Since
\[
x^{k+\half} = \VI_{\XX} \left(  \tilde F(x;x^k) + L_p \|x-x^k\|^{p-1} (x-x^k) \right),
\]
we have
\begin{equation} \label{subVI}
\langle \tilde F(x^{k+\half};x^k) + L_p \|x^{k+\half}-x^k\|^{p-1} (x^{k+\half}-x^k) , x - x^{k+\half}\rangle \ge 0,\,\, \forall x \in \XX.
\end{equation}
Denote $\gamma_k=L_p \|x^{k+\half}-x^k\|^{p-1}$.
Substituting $x=x^{k+1}$ in \eqref{subVI} we have
\begin{eqnarray}
&&\langle \tilde F(x^{k+\half};x^k)  ,  x^{k+1} - x^{k+\half}  \rangle \nonumber\\
&\ge& \gamma_k \langle x^{k+\half} - x^k , x^{k+\half} - x^{k+1}  \rangle \nonumber \\
&=& \frac{\gamma_k}{2} \left( \|x^{k+\half} - x^k\|^2 + \|x^{k+1}-x^{k+\half}\|^2 - \| x^{k+1}-x^k\|^2 \right) . \label{tildeF}
\end{eqnarray}

On the other hand, by the optimality condition at $x^{k+1}$ we have
\[
\langle F(x^{k+\half})  + \gamma_k (x^{k+1} - x^k) , x - x^{k+1} \rangle \ge 0 ,\mbox{ for all $x\in \XX$}.
\]
Hence,
\begin{eqnarray}
\langle F(x^{k+\half})  , x - x^{k+1} \rangle &\ge& \gamma_k \langle x^{k+1} - x^k, x^{k+1} - x \rangle  \nonumber \\
&=& \frac{\gamma_k}{2} \left(  \| x^{k+1} - x\|^2 + \| x^{k+1} - x^k\|^2 - \| x^k - x\|^2 \right) ,\, \mbox{ for all $x\in \XX$}. \label{F ineq}
\end{eqnarray}

Continue with the above inequality, for any given $x\in \XX$ we have
\begin{eqnarray*}
& & \frac{\gamma_k}{2} \left(  \| x^{k+1} - x\|^2 + \| x^{k+1} - x^k\|^2 - \| x^k - x\|^2 \right) \\
&\overset{\eqref{F ineq}}{\le} & \langle F(x^{k+\half})  , x - x^{k+1} \rangle \\
&=&  \langle F(x^{k+\half})  , x - x^{k+\half} \rangle + \langle F(x^{k+\half})  , x^{k+\half} - x^{k+1} \rangle \\
&=& \langle F(x^{k+\half})  , x - x^{k+\half} \rangle + \langle F(x^{k+\half}) - \tilde F(x^{k+\half};x^k) , x^{k+\half} - x^{k+1} \rangle + \langle \tilde F(x^{k+\half};x^k)  , x^{k+\half} - x^{k+1} \rangle \\
&\le& \langle F(x^{k+\half})  , x - x^{k+\half} \rangle + \| F(x^{k+\half}) - \tilde F(x^{k+\half};x^k) \| \cdot \| x^{k+\half} - x^{k+1} \|
+ \langle \tilde F(x^{k+\half};x^k)  , x^{k+\half} - x^{k+1} \rangle \\
&\le & \langle F(x^{k+\half})  , x - x^{k+\half} \rangle + \frac{\| F(x^{k+\half}) - \tilde F(x^{k+\half};x^k)\|^2}{2\gamma_k}  + \frac{\gamma_k \|  x^{k+\half} - x^{k+1} \|^2}{2}\\
& & + \langle \tilde F(x^{k+\half};x^k)  , x^{k+\half} - x^{k+1} \rangle \\
&\overset{\eqref{approx-oper-bound}}{\le} & \langle F(x^{k+\half})  , x - x^{k+\half} \rangle + \frac{\tau^2 L_p^2 \| x^{k+\half} - x^k\|^{2p} }{2\gamma_k}  + \frac{\gamma_k \|  x^{k+\half} - x^{k+1} \|^2}{2}\\
& & + \langle \tilde F(x^{k+\half};x^k)  , x^{k+\half} - x^{k+1} \rangle .
\end{eqnarray*}
Noticing that $\frac{\tau^2 L_p^2 \| x^{k+\half} - x^k\|^{2p}}{2\gamma_k}=\frac{\tau^2 \gamma_k \| x^{k+\half} - x^{k} \|^2}{2}$, and further using \eqref{tildeF} we derive from the above that
\begin{eqnarray*}
& & \frac{\gamma_k}{2} \left(  \| x^{k+1} - x\|^2 + \| x^{k+1} - x^k\|^2 - \| x^k - x\|^2 \right) \\
&\le& \langle F(x^{k+\half})  , x - x^{k+\half} \rangle + \frac{\tau^2 \gamma_k \| x^{k+\half} - x^{k} \|^2}{2} + \frac{\gamma_k \|  x^{k+\half} - x^{k+1} \|^2}{2} \\
&   & + \frac{\gamma_k}{2} \left[ - \|x^{k+\half} - x^k\|^2 - \|x^{k+1}-x^{k+\half}\|^2 + \| x^{k+1}-x^k\|^2 \right] .
\end{eqnarray*}
Canceling out terms, we simplify the above inequality into
\begin{equation} \label{iter}
\langle F(x^{k+\half})  , x^{k+\half} - x \rangle + \frac{\gamma_k}{2} \left(1-\tau^2 \right) \| x^{k+\half} - x^{k} \|^2 \le
\frac{\gamma_k}{2} \left[ \|x^{k} - x\|^2 - \|x^{k+1}-x\|^2  \right] .
\end{equation}
Consequently, by the monotonicity of $F$, we have
\begin{eqnarray*}
& & \langle F(x)  , x^{k+\half} - x \rangle + \frac{\gamma_k}{2} \left(1-\tau^2 \right) \| x^{k+\half} - x^{k} \|^2 \nonumber \\
& \le & \langle F(x^{k+\half})  , x^{k+\half} - x \rangle + \frac{\gamma_k}{2} \left(1-\tau^2 \right) \| x^{k+\half} - x^{k} \|^2 \nonumber \\
& \le &
\frac{\gamma_k}{2} \left[ \|x^{k} - x\|^2 - \|x^{k+1}-x\|^2  \right] .
\end{eqnarray*}
Dividing both sides by $\gamma_k$ yields
\begin{equation} \label{bound-1}
\frac{1}{\gamma_k}\left\langle  F(x)  ,  x^{k+\half} -  x \right\rangle + \frac{1}{2} \left(1-\tau^2 \right) \| x^{k+\half} - x^{k} \|^2 \le
\frac{1}{2} \left[ \|x^{k} - x\|^2 - \|x^{k+1}-x\|^2  \right] .
\end{equation}
% Denote
% \[
% \Gamma_N := \sum_{k=1}^N \frac{1}{\gamma_k}, \mbox{ and } \bar x_N := \frac{\sum_{k=1}^N \frac{1}{\gamma_k} x^{k+\half}}{\Gamma_N} \in \XX.
% \]

Summing up the inequality \eqref{bound-1} from $k=1$ to $N$, and dividing the resulting inequality on both sides by $\Gamma_N$ we obtain
\begin{equation} \label{bound-2}
\left\langle F(x)  , \bar x_N - x \right\rangle + \frac{1-\tau^2}{2\Gamma_N }  \sum_{k=1}^N \| x^{k+\half} - x^{k} \|^2 \le
\frac{\|x^{1} - x\|^2 }{2\Gamma_N}
\end{equation}
for any $x\in \XX$. Taking $x=x^*$ in \eqref{bound-2} yields
\begin{equation} \label{bound-3}
\sum_{k=1}^N \| x^{k+\half} - x^{k} \|^2 \le \frac{\|x^{1} - x^*\|^2 }{1-\tau^2} .
\end{equation}

The so-called {\it mean inequality}\/ maintains that for any positive sequence $\{ a_k>0 \mid k=1,2,...,N\}$ and any real value $r$, if we define
\[
M_r (a) := \left( \frac{1}{N} \sum_{k=1}^N a_k^r \right)^{\frac{1}{r}}
\]
then we have $M_{r_1}(a) \le M_{r_2}(a)$ for any $r_1\le r_2$.

Now, if we let $a_k := \| x^{k+\half} - x^{k} \|^{-(p-1)}$, then we have $M_{-\frac{2}{p-1}} (a) \le M_1 (a)$; that is
\[
\left( \frac{1}{N} \sum_{k=1}^N  \| x^{k+\half} - x^{k} \|^2 \right)^{ -\frac{p-1}{2} } \le \frac{1}{N}  \sum_{k=1}^N  \| x^{k+\half} - x^{k} \|^{-(p-1)}.
\]
Therefore,
\begin{eqnarray*}
\Gamma_N &=& \frac{1}{L_p} \sum_{k=1}^N \| x^{k+\half} - x^{k} \|^{-(p-1)} \ge \frac{N}{L_p} \left( \frac{1}{N} \sum_{k=1}^N  \| x^{k+\half} - x^{k} \|^2 \right)^{ -\frac{p-1}{2} }  \\
& \overset{\eqref{bound-3}}{\ge} & \frac{N^{\frac{p+1}{2}}}{L_p} \left( \frac{1-\tau^2}{\|x^{1} - x^* \|^2 } \right)^{ \frac{p-1}{2} } \ge \frac{N^{\frac{p+1}{2}}}{L_p} \left( \frac{1-\tau^2}{D^2 } \right)^{ \frac{p-1}{2} } .
\end{eqnarray*}

Then, \eqref{bound-2} leads to
\[
m(\bar x_N) \le \frac{D^2}{2\Gamma_N} \le \frac{D^2 L_p}{2N^{\frac{p+1}{2}} \left( \frac{1-\tau^2}{D^2 } \right)^{ \frac{p-1}{2} }}
= \frac{D^{p+1}L_p}{2 \left(1-\tau^2\right)^{\frac{p-1}{2}} N^{\frac{p+1}{2}}} = O\left(\frac{1}{N^{\frac{p+1}{2}}}\right).
\]
\end{proof}

Theorem \ref{thm:are-global} implies that the proposed ARE generated $\bar x_N$ such that $m(\bar x_N)\le \epsilon$ with iteration complexity $\mathcal{O}\left(1/\epsilon^{\frac{2}{p+1}}\right)$. This matches the lower bound $\Omega\left(1/\epsilon^{\frac{2}{p+1}}\right)$ established in \cite{lin2022perseus}, hence optimal. The concept of solving a $(p+1)^\T$-order regularized VI subproblem is also proposed in \cite{adil2022optimal} and \cite{lin2022perseus}, therefore there is no need for an additional bisection subroutine in each iteration. The major difference between ARE and the method proposed in \cite{adil2022optimal} is that ARE uses a more general approximation operator $\tilde F(x;x^k)$ in the aforementioned regularized VI subproblem, which generalizes the Taylor approximation proposed in \cite{adil2022optimal}, and we provide a unified analysis as long as \eqref{approx-oper-bound} is satisfied. The flexibility of not being restricted to Taylor approximation is demonstrated in Section \ref{sec:subproblem-2} and Section \ref{sec:structured-are}, where examples are given for applying ARE with non-Taylor approximation $\tilde F(x;x^k)$ to problems when the original operator exhibits composite structure $F(x)=H(x)+G(x)$ or more generally $F(x)=H(G(x))$.

\subsection{Solving strongly monotone VI with ARE-restart}
\label{sec:ARE-global-strong}
% While it is possible to apply the same iterative procedure for first-order methods when the VI is strongly monotone instead of monotone (e.g. extra-gradient method compared to mirror-prox method) to achieve the optimal convergence rate $\left(1-\frac{\mu}{L}\right)$, when it comes to improved convergence rate for $p>1$ in \eqref{conv-monotone}, it is often required to incorporate a {\it restart} procedure into the original iterations for monotone VI. See also \cite{ostroukhov2020tensor} for restarting the higher-order mirror-prox method in \cite{bullins2020higher}. 
While Theorem \ref{thm:are-global} establishes the optimal sublinear convergence for monotone VI, we shall incorporate a {\it restart} procedure into ARE \eqref{are-update} to further establish an improved linear convergence for strongly monotone VI. Similar restarting procedure is also seen in previous work \cite{ostroukhov2020tensor, lin2022perseus} for establishing the linear convergence. Below we give a detailed analysis for restarting the $p^\T$-order ARE \eqref{are-update}, referred to as {\it ARE-restart}.

The ARE-restart works in {\it epochs}. That is, for each epoch $m$, where $m=1,2,...$, a number of iterative updates \eqref{are-update} is performed and the output is set as the initial iterate at the start of the next epoch.

Let $N_m$ denote the number of iterations performed in $m^\T$ epoch, $m=1,2,...$. After each $N_m$ iterations (of ARE), we restart ($x^1\leftarrow \bar{x}_{N_m}$) and proceed to $(m+1)^\T$ epoch. The output of epoch $m$ is defined as:
\[
\bar x_{N_m} := \frac{\sum_{k=1}^{N_m} \frac{1}{\gamma_k} x^{k+\half}}{\Gamma_{N_m}} \in \XX, \mbox{ and } \Gamma_{N_m} := \sum_{k=1}^{N_m} \frac{1}{\gamma_k}.
\]

Denote $D_0=\|x^1-x^*\|$ as the distance to the solution from the very first initial iterate before any restart and note that $D_0\le D$. Let $0<\delta<1$ be a constant independent of the problem. Let us fix the iterations in each epoch:
\begin{eqnarray}
N_1=N_2=\cdots =N = \left(\frac{L_p}{2\delta\mu}\right)^{\frac{2}{p+1}}\left(\frac{D_0^2}{1-\tau^2}\right)^{\frac{p-1}{p+1}}.\label{iteration-in-epoch}
\end{eqnarray}

From the analysis in Theorem \ref{thm:are-global}, we can first reach \eqref{iter}, where by using the strong monotonicity we have:
\begin{eqnarray}
& & \mu\|x^{k+\half}-x\|^2+\langle F(x)  , x^{k+\half} - x \rangle + \frac{\gamma_k}{2} \left(1-\tau^2 \right) \| x^{k+\half} - x^{k} \|^2 \nonumber \\
& \le & \langle F(x^{k+\half})  , x^{k+\half} - x \rangle + \frac{\gamma_k}{2} \left(1-\tau^2 \right) \| x^{k+\half} - x^{k} \|^2 \nonumber \\
& \le &
\frac{\gamma_k}{2} \left[ \|x^{k} - x\|^2 - \|x^{k+1}-x\|^2  \right] .\label{iter-sm}
\end{eqnarray}
Taking $x=x^*$ and sum the inequality from $k=1$ to $N$:
\begin{eqnarray}
\mu\Gamma_N\|\bar{x}_N-x^*\|^2&\le&\sum\limits_{k=1}^N\frac{\mu}{\gamma_k}\|x^{k+\half}-x^*\|^2\nonumber\\
&\le& \frac{1}{2}\|x^1-x^*\|^2-\frac{1}{2}\|x^{N+1}-x^*\|^2-\frac{1-\tau^2}{2}\sum\limits_{k=1}^N\|x^{k+\half}-x^k\|^2, \label{tele-sm}
\end{eqnarray}
where the first inequality is due to the convexity of the squared norm function.

Now consider the first epoch, inequality \eqref{tele-sm} implies:
\begin{eqnarray}
\|\bar{x}_{N_1}-x^*\|^2\le \frac{1}{2\mu\Gamma_{N_1}}D_0^2,\nonumber
\end{eqnarray}
where
\begin{eqnarray}
\Gamma_{N_1}\ge \frac{N^{\frac{p+1}{2}}}{L_p}\left(\frac{1-\tau^2}{\|x^1-x^*\|^2}\right)^{\frac{p-1}{2}}=\frac{N^{\frac{p+1}{2}}}{L_p}\left(\frac{1-\tau^2}{D_0^2}\right)^{\frac{p-1}{2}}.\nonumber
\end{eqnarray}
Therefore,
\begin{eqnarray}
\|\bar{x}_{N_1}-x^*\|^2\le \frac{1}{2\mu\Gamma_{N_1}}D_0^2\le \frac{L_p}{2\mu}\frac{D_0^{p+1}}{(1-\tau^2)^{\frac{p-1}{2}}}\frac{1}{N^{\frac{p+1}{2}}}=\delta\cdot D_0^2.\nonumber
\end{eqnarray}
Now, in the second epoch, we take $x^1\leftarrow\bar{x}_{N_1}$. Similarly, we have:
\begin{eqnarray}
\|\bar{x}_{N_2}-x^*\|^2\le \frac{1}{2\mu\Gamma_{N_2}}\|\bar{x}_{N_1}-x^*\|^2\le\frac{\delta}{2\mu\Gamma_{N_2}}D_0^2,\nonumber
\end{eqnarray}
where the lower bound of $\Gamma_{N_2}$ can also be estimated from \eqref{tele-sm}, with an improved distance to solution $\|x^1-x^*\|^2=\|\bar{x}_{N_1}-x^*\|^2\le\delta D_0^2$:
\begin{eqnarray}
\Gamma_{N_2}\ge \frac{N^{\frac{p+1}{2}}}{L_p}\left(\frac{1-\tau^2}{\|\bar{x}_{N_1}-x^*\|^2}\right)^{\frac{p-1}{2}}\ge\frac{N^{\frac{p+1}{2}}}{L_p}\left(\frac{1-\tau^2}{\delta D_0^2}\right)^{\frac{p-1}{2}}.\nonumber
\end{eqnarray}
Note that in the second epoch, the lower bound of $\Gamma_{N_2}$ is improved by $\left(\frac{1}{\delta}\right)^{\frac{p-1}{2}}$. Then we have:
\begin{eqnarray}
\|\bar{x}_{N_2}-x^*\|^2\le \frac{\delta}{2\mu\Gamma_{N_2}}D_0^2\le \frac{L_p}{2\mu}\frac{D_0^{p+1}}{(1-\tau^2)^{\frac{p-1}{2}}}\frac{\delta^{\frac{p+1}{2}}}{N^{\frac{p+1}{2}}}=\delta^{\frac{p+3}{2}}\cdot D_0^2.\nonumber
\end{eqnarray}
Note that after second epoch, the distance is not decreased by a factor of $\delta^2$ but a factor of $\delta^{\frac{p+3}{2}}$ instead. {This is because while performing $N$ iterations in one epoch provides a decrease of $\delta$ in terms of the original distance $D_0$, starting from an iterate $x^1=\bar{x}^{N_1}$ in the second epoch provides an additional decrease of $\delta^{\frac{p+1}{2}}$ due to a better bound for $\|\bar{x}_{N_1}-x^*\|^2$ and $\Gamma_{N_2}$.} Now continue considering the third epoch:
\begin{eqnarray}
\|\bar{x}_{N_3}-x^*\|^2\le \frac{1}{2\mu\Gamma_{N_3}}\|\bar{x}_{N_2}-x^*\|^2\le\frac{\delta^{\frac{p+3}{2}}}{2\mu\Gamma_{N_3}}D_0^2,\nonumber
\end{eqnarray}
where
\begin{eqnarray}
\Gamma_{N_3}\ge \frac{N^{\frac{p+1}{2}}}{L_p}\left(\frac{1-\tau^2}{\|\bar{x}_{N_2}-x^*\|^2}\right)^{\frac{p-1}{2}}\ge\frac{N^{\frac{p+1}{2}}}{L_p}\left(\frac{1-\tau^2}{\delta^{\frac{p+3}{2}} D_0^2}\right)^{\frac{p-1}{2}}.\nonumber
\end{eqnarray}
Therefore,
\begin{eqnarray}
\|\bar{x}_{N_3}-x^*\|^2\le \frac{\delta^{\frac{p+3}{2}}}{2\mu\Gamma_{N_3}}D_0^2\le \frac{L_p}{2\mu}\frac{D_0^{p+1}}{(1-\tau^2)^{\frac{p-1}{2}}}\frac{\delta^{\frac{(p+3)(p+1)}{4}}}{N^{\frac{p+1}{2}}}=\delta^{\frac{(p+3)(p+1)}{4}+1}\cdot D_0^2.\nonumber
\end{eqnarray}
To summarize, after $m$ epochs, we have
\begin{eqnarray}
\|\bar{x}_{N_m}-x^*\|^2\le \delta^{t_{m}}\cdot D_0^2,\nonumber
\end{eqnarray}
where
\begin{eqnarray}
t_m = t_{m-1}\cdot\frac{p+1}{2}+1,\quad t_1 = 1.\nonumber
\end{eqnarray}
Then we have
\begin{eqnarray}
\|\bar{x}_{N_m}-x^*\|^2\le \delta^{t_{m}}\cdot D_0^2\le \delta^{\left(\frac{p+1}{2}\right)^{m-1}}\cdot D_0^2\le \delta^{\left(\frac{p+1}{2}\right)^{m-1}}\cdot D^2.\nonumber
\end{eqnarray}
That is, the total number of epochs required to have $\|\bar{x}_{N_m}-x^*\|^2\le\epsilon$ is given by
\begin{eqnarray}
\log_{\frac{p+1}{2}}\log_{\frac{1}{\delta}}\frac{D^2}{\epsilon},\label{super-linear-global}
\end{eqnarray}
for $p>1$, a superlinear rate for the {\it epochs}. For $p=1$, the log is one layer and we only have linear convergence. Note that, however, \eqref{super-linear-global} is only the number of {\it epochs} needs to be run, and for each epoch a fixed number of $N$ iterations is still performed, so the total iteration complexity is:
\begin{eqnarray}
\left(\frac{L_p}{2\delta\mu}\right)^{\frac{2}{p+1}}\left(\frac{D_0^2}{1-\tau^2}\right)^{\frac{p-1}{p+1}}\log_{\frac{p+1}{2}}\log_{\frac{1}{\delta}}\frac{D^2}{\epsilon}\nonumber
\end{eqnarray}
for $p>1$. For simplicity, we can take $\delta=\frac{1}{2}$ and replace $D_0$ with $D$ in the number of iterations $N$ in one epoch, which gives the complexity:
\[
\mathcal{O}\left(\left(\frac{L_p}{\mu}\right)^{\frac{2}{p+1}}\left(D^2\right)^{\frac{p-1}{p+1}}\log_{\frac{p+1}{2}}\log_{2}\frac{D^2}{\epsilon}\right).
\]
The result is summarized in the next theorem.
\bigskip
\begin{theorem}[Global convergence of ARE: Strongly monotone VI]
\label{thm:are-global-strong}
Let $\{x^k\}_{k\ge1}$ and $\{x^{k+0.5}\}_{k\ge1}$ be generated by ARE \eqref{are-update} and suppose $F(\cdot)$ is strongly monotone with $\mu>0$ and $\tilde{F}(x;x^k)$ is such that \eqref{approx-oper-bound} holds. By restarting ARE after each $N_i=N$ iterations in epoch $i$, where $N$ is given by \eqref{iteration-in-epoch}, the total number of epochs $m$ required to obtain an output $\bar{x}_{N_m}$ such that $\|\bar{x}_{N_m}-x^*\|^2\le\epsilon$ is given by:
\begin{eqnarray}
\left\{
\begin{array}{ll}
     \mathcal{O}\left(\log_{\frac{p+1}{2}}\log_{2}\frac{D^2}{\epsilon}\right),\quad & p>1, \\
     &\\
     \mathcal{O}\left(\log_2\frac{D^2}{\epsilon}\right),\quad &p=1.
\end{array}
\right.\nonumber
\end{eqnarray}
The total iteration complexity $mN$ is given by:
\begin{eqnarray}
\left\{
\begin{array}{ll}
     \mathcal{O}\left(\left(\frac{L_p}{\mu}\right)^{\frac{2}{p+1}}\left(D^2\right)^{\frac{p-1}{p+1}}\log_{\frac{p+1}{2}}\log_{2}\frac{D^2}{\epsilon}\right),\quad & p>1, \\
     &\\
     \mathcal{O}\left(\frac{L_1}{\mu}\log_2\frac{D^2}{\epsilon}\right),\quad &p=1.
\end{array}
\right.\label{global-restart-ARE}
\end{eqnarray}

\end{theorem}

Through a careful analysis of the restarting procedure, Theorem \ref{thm:are-global-strong} shows that by restarting ARE when $F(x)$ is strongly monotone, the optimal iteration complexity is achievable for $p=1$ and the improved iteration complexity is obtained for $p>1$, as summarized in \eqref{global-restart-ARE}. We note that the total number of epochs (or the number of restarting) is only of the order $\mathcal{O}\left(\log_{\frac{p+1}{2}}\log_2\frac{D^2}{\epsilon}\right)$, which is an improved bound compared to $\mathcal{O}\left(\log_2\frac{D^2}{\epsilon}\right)$ established in \cite{ostroukhov2020tensor, lin2022perseus}. This implies that the output iterate $\bar x_{N_j}$ after each epoch for $j=1,...,m$ converges towards $x^*$ at a superlinear rate, and the reason being that the lower bound for the averaging parameter $\Gamma_{N_j}$ is improved by an order of $\frac{p+1}{2}$. 

\section{The Local Convergence Analysis of ARE}
\label{sec:ARE-local}
In this section we shall analyze the local convergence behavior of ARE for strongly monotone $F$ (i.e.\ $\mu>0$) when $p>1$. A pure Newton method typically exhibits local quadratic convergence in optimization, and the same has been shown for VI \cite{taji1993globally, huang2022cubic}. In \cite{ostroukhov2020tensor}, while the global iterations proceed with restarting higher-order mirror-prox method \cite{bullins2020higher} and an iteration complexity similar to \eqref{global-restart-ARE} is established, the local iterations are performed by adopting CRN-SPP \cite{huang2022cubic} to obtain quadratic convergence. The local superlinear convergence is further improved in \cite{lin2022perseus} by restarting Perseus and in \cite{jiang2022generalized}, to the order $\frac{p+1}{2}$. In the following analysis, we show that in the $p^\T$-order ARE where \eqref{approx-oper-bound} is satisfied, % with $p$, 
then the %local convergence is instead 
$p^\T$-order local superlinear convergence rate holds, 
which is an improvement compared to existing work in the literature.

We first show that for the ARE update \eqref{are-update}, $\|x^{k+0.5}-x^*\|$ converges to zero $p^\T$
-order superlinearly compared to $\|x^k-x^*\|$, as long as $\|x^{k+0.5}-x^k\|$ is sufficiently small. By the definition of a VI solution for update $x^{k+0.5}$:
\begin{equation} \nonumber
\langle \tilde F(x^{k+\half};x^k) + L_p \|x^{k+\half}-x^k\|^{p-1} (x^{k+\half}-x^k) , x - x^{k+\half}\rangle \ge 0,\,\, \forall x \in \XX.
\end{equation}
Then
\begin{eqnarray}
&&\langle F(x^{k+0.5}),x^{k+0.5}-x\rangle\nonumber\\
&\le& \langle \tilde{F}(x^{k+0.5};x^k)-F(x^{k+0.5})+L_p\|x^{k+0.5}-x^k\|^{p-1}(x^{k+0.5}-x^k),x-x^{k+0.5}\rangle\nonumber\\
&\le& \left\|\tilde{F}(x^{k+0.5};x^k)-F(x^{k+0.5})+L_p\|x^{k+0.5}-x^k\|^{p-1}(x^{k+0.5}-x^k)\right\|\cdot\|x-x^{k+0.5}\|\nonumber\\
&\le& \left(\left\|\tilde{F}(x^{k+0.5};x^k)-F(x^{k+0.5})\right\|+\left\|L_p\|x^{k+0.5}-x^k\|^{p-1}(x^{k+0.5}-x^k)\right\|\right)\cdot\|x-x^{k+0.5}\|\nonumber\\
&\overset{\eqref{approx-oper-bound}}{\le}& (1+\tau)L_p\|x^{k+0.5}-x^k\|^p\cdot\|x-x^{k+0.5}\|.\label{local-step-1}\nonumber
\end{eqnarray}
Take $x=x^*$ and use the strong monotonicity of $F$:
\begin{eqnarray}
&&\mu\|x^{k+0.5}-x^*\|^2+\langle F(x^*),x^{k+0.5}-x^*\rangle\le\langle F(x^{k+0.5}),x^{k+0.5}-x^*\rangle\nonumber\\
&\le& (1+\tau)L_p\|x^{k+0.5}-x^k\|^p\cdot\|x^*-x^{k+0.5}\|,\label{local-step-2}\nonumber
\end{eqnarray}
then we have
\begin{eqnarray}
\|x^{k+0.5}-x^*\|\le\frac{(1+\tau)L_p}{\mu}\|x^{k+0.5}-x^k\|^p.\label{p-super-1}
\end{eqnarray}
Now, by the same analysis from \eqref{subVI}-\eqref{iter} and take $x=x^*$, we have:
\begin{equation} \nonumber
\langle F(x^{k+\half})  , x^{k+\half} - x^* \rangle + \frac{\gamma_k}{2} \left(1-\tau^2 \right) \| x^{k+\half} - x^{k} \|^2 \le
\frac{\gamma_k}{2} \left[ \|x^{k} - x^*\|^2 - \|x^{k+1}-x^*\|^2  \right] .
\end{equation}
Noticing that $\langle F(x^{k+\half})  , x^{k+\half} - x^* \rangle\ge \langle F(x^*),x^{k+0.5}-x^*\rangle\ge0$, the above inequality implies
\begin{eqnarray}
(1-\tau^2)\|x^{k+0.5}-x^k\|^2\le \|x^k-x^*\|^2.\label{p-super-2}
\end{eqnarray}
Combining \eqref{p-super-1} and \eqref{p-super-2} gives the {\it $p^\T$-order superlinear convergence}:
\begin{eqnarray}
\|x^{k+0.5}-x^*\|\le\frac{(1+\tau)L_p}{\mu(1-\tau^2)^{\frac{p}{2}}}\|x^{k}-x^*\|^p.\label{p-super-3}
\end{eqnarray}
Note, however, that the inequality \eqref{p-super-3} only holds within each iteration and $x^k$ in general is not converging towards $x^*$ $p^\T$-order superlinerly if a subsequent extra-gradient update is performed as in \eqref{are-update}. In fact, once the local convergence behavior is observed, the extra-gradient update should be suppressed and the algorithm should accept $x^{k+0.5}$ as the next iterate. We shall denote
\begin{eqnarray}
x^{k+1}:=x^{k+0.5} &:=& \VI_{\XX} \left(  \tilde F(x;x^k) + L_p \|x-x^k\|^{p-1} (x-x^k) \right)\label{ar-update}
\end{eqnarray}
as {\it Approximation-based Regularized (AR)} update. Algorithm \ref{alg:restart-ARE-local} incorporates the above decision process into ARE-restart proposed in \ref{sec:ARE-global-strong}, such that both the improved global iteration complexity \eqref{global-restart-ARE} and the local $p^\T$-order superlinear convergence are attained.

\begin{algorithm}[ht!]
	\caption{ARE-Restart with Local Superlinear Convergence}
	\begin{algorithmic}[1]
		\Require $x^1\in \XX$, $0<\alpha<1$, $D\ge\|x^1-x^*\|$, an inner iteration number
		\[
		N = \left\lceil\left(\frac{L_p}{\mu}\right)^{\frac{2}{p+1}}\left(\frac{D^2}{1-\tau^2}\right)^{\frac{p-1}{p+1}}
        \right\rceil.
		\]
		\State {\bf Step 0:} Set $k:=1$.
		\State {\bf Step 1:} Let
		\[
        x^{k+\half} := \VI_{\XX} \left(  \tilde F(x;x^k) + L_p \|x-x^k\|^{p-1} (x-x^k) \right) .
        \]
        If
        \begin{eqnarray}
        \|x^{k+\half} - x^k\|^{p-1} \le \frac{\alpha \sqrt{1-\tau^2}}{1+\tau} \frac{\mu}{L_p}\label{local-condition}
        \end{eqnarray}
        then $x^{k+1}:=x^{k+0.5}$ ({AR update}), set $k:=k+1$, and return to {\bf Step 1}. Otherwise, go to {\bf Step 2}.
		\State {\bf Step 2:} Let
		\[
        x^{k+1} := \mbox{arg}\min_{x \in \XX} \, \langle F(x^{k+\half}) , x - x^k \rangle + \frac{L_p\|x^{k+\half}-x^k\|^{p-1}}{2} \|x-x^k\|^2.\quad\mbox{(ARE-update)}
        \]
        If $k=N$, let
        \[
        \Gamma_{N} := \sum_{k=1}^{N} \frac{1}{\gamma_k}, \mbox{ and } \bar x_{N} := \frac{\sum_{k=1}^{N} \frac{1}{\gamma_k} x^{k+\half}}{\Gamma_{N}},
        \]
        set $x^1:=\bar x_N$, and return to {\bf Step 0}. Otherwise, set $k:=k+1$ and return to {\bf Step 1}.
	\end{algorithmic}
	\label{alg:restart-ARE-local}
\end{algorithm}

To verify the local convergence of Algorithm \ref{alg:restart-ARE-local}, we are left to show that once condition \eqref{local-condition} is satisfied and AR update (i.e.\ $x^{k+1}:=x^{k+0.5}$) is accepted in Step 1, the algorithm will continue repeating Step 1 to obtain
\begin{eqnarray}
\|x^{k+1}-x^*\|\le\frac{(1+\tau)L_p}{\mu(1-\tau^2)^{\frac{p}{2}}}\|x^{k}-x^*\|^p.\nonumber
\end{eqnarray}
Indeed, from the previous analysis with $x^{k+0.5}$ replaced with $x^{k+1}$ in \eqref{p-super-1} and \eqref{p-super-2}, we have
\begin{eqnarray}
\|x^{k+1} - x^k \| &\le& \frac{1}{\sqrt{1-\tau^2}} \|x^k - x^*\| ,\nonumber \\ 
\|x^{k+1}-x^*\| &\le& (1+\tau)\frac{L_p}{\mu} \, \|x^{k+1}-x^k\|^p ,\nonumber
\end{eqnarray}
which implies
\[
\|x^{k+2}-x^{k+1} \| \le \frac{1}{\sqrt{1-\tau^2}}\, \|x^{k+1}-x^*\| 
\le \frac{1+\tau}{\sqrt{1-\tau^2}} \frac{L_p}{\mu} \, \|x^{k+1}-x^k\|^p\le \alpha\|x^{k+1}-x^k\|,
\]
where the last inequality holds due to the condition \eqref{local-condition}. Therefore, $\{\|x^{k+1}-x^k\|\}$ becomes a contracting sequence once AR update is accepted, and Algorithm \ref{alg:restart-ARE-local} will repeat Step 1 until the designated total iteration number.

We summarize the iteration complexity of Algorithm \ref{alg:restart-ARE-local} in the next theorem.

\begin{theorem}
\label{thm:restart-are-local}
Let $\{x^k\}_{k\ge1}$ and $\{x^{k+0.5}\}_{k\ge1}$ be generated by Algorithm \ref{alg:restart-ARE-local} and suppose $F(\cdot)$ is strongly monotone with $\mu>0$ and $\tilde{F}(x;x^k)$ is such that \eqref{approx-oper-bound} holds with $p>1$. The total iteration complexity to reach $\|x^k-x^*\|\le\epsilon$ for some $\epsilon>0$ is given by:
\begin{eqnarray}
\tilde{\mathcal{O}}\left(\left(\frac{L_p}{\mu}\right)^{\frac{2}{p+1}}(D^2)^{\frac{p-1}{p+1}}+\log_{p}\log_2\frac{1}{\epsilon}\right).\label{local-iter-complexity}
\end{eqnarray}
\end{theorem}
\begin{proof}
We omit the proof of the local iteration complexity $\log_p\log_2\frac{1}{\epsilon}$ in view of the earlier arguments. Note that we have used $\tilde{\mathcal{O}}$ to suppress the logarithmic part in the first term of \eqref{local-iter-complexity}. Since Algorithm \ref{alg:restart-ARE-local} adopts ARE-restart as global iterations, by Theorem \ref{thm:are-global-strong}, the iteration complexity requires to reach $\|x^{k}-x^*\|^2\le\hat{\epsilon}$ for some $\hat{\epsilon}>0$ is 
\[
\mathcal{O}\left(\left(\frac{L_p}{\mu}\right)^{\frac{2}{p+1}}\left(D^2\right)^{\frac{p-1}{p+1}}\log_{\frac{p+1}{2}}\log_{2}\frac{D_0^2}{\hat{\epsilon}^2}\right).
\]
In view of \eqref{p-super-2}, let $\hat{\epsilon}:=(1-\tau^2)\left(\frac{\alpha\sqrt{1-\tau^2}}{1+\tau}\cdot\frac{\mu}{L_p}\right)^{\frac{2}{p-1}}$, then we have:
\[
\|x^{k+0.5}-x^k\|^{p-1}\le \left(\frac{\|x^k-x^*\|^2}{1-\tau^2}\right)^{\frac{p-1}{2}}\le \left(\frac{\hat{\epsilon}}{1-\tau^2}\right)^{\frac{p-1}{2}}\le\frac{\alpha \sqrt{1-\tau^2}}{1+\tau} \frac{\mu}{L_p}.
\]

\end{proof}

As far as we know, the results on local superlinear convergence for higher-order VI methods are still quite limited in the literature. %The authors in
In \cite{huang2022cubic}, the authors establish quadratic convergence for strongly-convex-strongly-concave saddle point problem with a cubic regularized method CRN-SPP ($p=2$). Such local quadratic convergence result is also adopted by \cite{ostroukhov2020tensor} for the general $p^\T$-order method. This rate is further improved to $\frac{p+1}{2}$ in \cite{jiang2022generalized, lin2022perseus}. However, we show that for the $p^\T$-order ARE with $p>1$, the local superlinear convergence is of the order $p$, achieved by the AR update \eqref{ar-update}. As shown in Algorithm \ref{alg:restart-ARE-local}, there is an implementable criterion \eqref{local-condition} to determine whether to reject the extra step and continue to converge superlinearly. We also note the difference between the results in Theorem \ref{thm:are-global-strong} and Theorem \ref{thm:restart-are-local}. In Theorem \ref{thm:are-global-strong}, the iterates converge superlinerly after each {\it epoch}, which still requires $\mathcal{O}\left(\left(\frac{L_p}{\mu}\right)^{\frac{2}{p+1}}\left(D\right)^{\frac{p-1}{p+1}}\right)$ inner iterations between restarts. On the other hand, when $\eqref{local-condition}$ is satisfied and the algorithm starts to perform only the AR updates, the iterates start to converge superlinearly after each {\it iteration}. The overall iteration complexity is then given in \eqref{local-iter-complexity}.

\section{Solving Regularized VI Subproblem with $p=2$}
\label{sec:subproblem}

In the previous sections, we have presented global and local iteration complexity analysis for ARE \eqref{are-update}. We show in Theorem \ref{thm:are-global} that the proposed simple update form of ARE guarantees the same order of improved iteration complexity as \cite{jiang2022generalized, adil2022optimal, lin2022perseus} for $p>1$ under the monotone case, which is also optimal due to the lower bound established in \cite{lin2022perseus}. For strongly monotone VI, we show that by restarting ARE, the iterates after each epochs converge at a superlinear rate with per-epoch cost $\left(\frac{L_p}{\mu}\right)^{\frac{2}{p+1}}(D^2)^{\frac{p-1}{p+2}}$. We further show that by imposing an additional condition \eqref{local-condition} before performing the extra-step update, the local $p^\T$-order superlinear convergence is guaranteed. 

The aforementioned results are derived based on the assumption that the first step of ARE, which involves solving an approximation regularized VI subproblem
\begin{eqnarray}
x^{k+\half} := \VI_{\XX} \left(  \tilde F(x;x^k) + L_p \|x-x^k\|^{p-1} (x-x^k) \right),\label{subproblem-p}
\end{eqnarray}
can be efficiently performed with incomparable cost to the overall iterations. While this assumption is commonly made for higher-order methods especially for $p\ge2$ in order to focus on analyzing the iteration complexities \cite{jiang2022generalized, adil2022optimal, lin2022perseus}, we shall devote this section to the discussion on certain details for solving such subproblem \eqref{subproblem-p}. The rest of the section will focus on the case when $p=2$ and the approximation mapping $\tilde F(x;x^k)$ is the corresponding Taylor approximation of $F(x)$, which is arguably most practical in the higher-order regime and admits some meaningful simplification and/or transformation of the subproblem.

Let us rewrite the subproblem \eqref{subproblem-p} in the following form:
\begin{eqnarray}
\mbox{(S)}\quad&& x^{k+0.5}= \VI_{\XX}\left(F(x^k)+\nabla F(x^k)(x-x^k)+L_2\|x-x^k\|(x-x^k)\right)\label{subproblem-p-2}\\
&\Longleftrightarrow& \mbox{find $x^{k+0.5}$ s.t.}\nonumber\\
&&\langle F(x^k)+\nabla F(x^k)(x^{k+0.5}-x^k)+L_2\|x^{k+0.5}-x^k\|(x^{k+0.5}-x^k),x-x^{k+0.5}\rangle\ge0,\nonumber
\end{eqnarray}
for all $x\in\XX$. We shall refer to \eqref{subproblem-p-2} as subproblem (S) and discuss two types of methods for solving it. In essence, the two types of methods both reduce the original subproblem (S) to another subproblem that can be more easily solved, and by solving the latter subproblem iteratively, we are able to obtain a(n) (approximated) solution to (S). 

\subsection{Reduction to VI subproblem with linear mapping}
\label{sec:subproblem-1}
In view of \eqref{subproblem-p-2}, the operator in VI subproblem (S) takes the form of the sum of a linear operator $F(x^k)+\nabla F(x^k)(x-x^k)$ and a non-linear operator $L_2\|x-x^k\|(x-x^k)$, where the latter is the gradient mapping of the cubic regularization term and in general makes the original VI problem difficult to solve efficiently. The first type of methods then aim to reduce \eqref{subproblem-p-2} to an easier VI problem with linear operator only, by {\it parameterizing} the solution $x^{k+0.5}$ as following:
\begin{eqnarray}
\mbox{(SS1)}\quad x^{k+0.5}(\lambda):=\VI_{\XX}\left(F(x^k)+\nabla F(x^k)(x-x^k)+\lambda(x-x^k)\right),\label{lambda-subproblem}
\end{eqnarray}
with the goal of finding
\begin{eqnarray}
\lambda = L_2\|x^{k+0.5}(\lambda)-x^k\|.\label{lambda-eqn}
\end{eqnarray}
Note that subproblem (SS1) given in \eqref{lambda-subproblem} is now a VI with linear operator $\nabla F(x^k)+\lambda I$. In particular, when $\XX=\mathbb{R}^n$ (i.e.\ unconstrained), $x^{k+0.5}(\lambda)$ admits the closed-form expression:
\begin{eqnarray}
&&F(x^k)+\left(\nabla F(x^k)+\lambda I\right)(x^{k+0.5}(\lambda)-x^k)=0,\nonumber\\
&\Longleftrightarrow& x^{k+0.5}(\lambda)=x^k-\left(\nabla F(x^k)+\lambda I\right)^{-1}F(x^k).\nonumber
\end{eqnarray}
This enables us to solve \eqref{lambda-eqn} using the next equation system with one-dimensional variable $\lambda$ via Newton method:
\begin{eqnarray}
f(\lambda):=\lambda^2-L_2^2\|x^{k+0.5}(\lambda)-x^k\|^2=0.\nonumber
\end{eqnarray}
In each iteration of the Newton method, it is then required to calculate the Jacobian of $f(\lambda)$. We shall omit the implementation details here and refer the interested readers to Section 4 in \cite{huang2022cubic}, which describes a more involved decomposition for calculating the Jacobian under the saddle-point problem setting (where $F(x^k)$ is the gradient descent ascent field of the saddle function).

For the case where \eqref{lambda-subproblem} is constrained with general closed convex set, \cite{monteiro2012iteration} proposed a {\it bisection} procedure for solving \eqref{lambda-eqn}. Similar ideas are also adopted in the bisection subroutine in \cite{bullins2020higher, jiang2022generalized}. We briefly summarize the underlying concept of such method and refer the interested readers to \cite{monteiro2012iteration} for analysis and implementation details. Instead of solving the equality constraint \eqref{lambda-eqn}, one can extend it to inequality constraints:
\begin{eqnarray}
\frac{L_2}{2}\|x^{k+0.5}(\lambda)-x^k\|\le \lambda\le 2L_2\|x^{k+0.5}(\lambda)-x^k\|.\label{const-lambda}
\end{eqnarray}
Note a similar constraint in \cite{bullins2020higher} for $p=2$. It is shown in \cite{monteiro2012iteration} that $\lambda$ satisfying constraint \eqref{const-lambda} lies in a closed interval: $\lambda\in[t_-,t_+]$ for some $t_+>t_->0$, whose range $[t_-,t_+]\subset[\alpha_-,\alpha_+]$ can be determined through solving (SS1) \eqref{lambda-subproblem} once with initialized $\lambda_0$. Using a bisection method which uses $\lambda_+=\sqrt{\alpha_-\alpha_+}$, the total complexity of solving (SS1) \eqref{lambda-subproblem} with \eqref{const-lambda} being satisfied is given by:
\[
\log\left(\frac{\log(\alpha_+/\alpha_-)}{\log(t_+/t_-)}\right).
\]
Therefore, solving the subproblem (S) boils down to how to solve (SS1) efficiently at each iteration of the bisection method proposed in \cite{monteiro2012iteration}. Since the VI operator for solving (SS1) is linear, if the constraint set $\XX$ takes simpler forms such as $\XX:=\mathbb{R}^n_+$, \eqref{lambda-subproblem} can be reduced to a {\it linear complementarity problem} (LCP), which then can be solved efficiently by using, for example, interior point method \cite{facchinei2007finite}.

\subsection{Reduction to gradient projection}
\label{sec:subproblem-2}
In the second type of method, we propose an alternative procedure to solve (S), which applies a first-order iterative method (an inner loop) to solve the VI problem \eqref{subproblem-p-2} for an approximated solution for $x^{k+0.5}$. Let us define the following operator
\begin{eqnarray}
F'(x;x^k):=F(x^k)+\nabla F(x^k)(x-x^k)+L_2\|x-x^k\|(x-x^k).\label{subproblem-operator}
\end{eqnarray}
A naive way to implement the inner loop is to directly apply, for example, the extra-gradient method to solve $\VI_{\XX}(F'(x;x^k))$. The potential issue lies in the fact that $F'(x;x^k)$ is in general not Lipschitz continuous over the whole constraint $\XX$ due to the mapping $L_2\|x-x^k\|(x-x^k)$. Therefore, no guarantee on the performance of the inner loop can be established.

In order to implement a more efficient procedure for solving \eqref{subproblem-p-2} (or succinctly $\VI_{\XX}(F'(x;x^k))$), we first discuss a specific instance of the ARE update introduced in Section \ref{sec:ARE-global} with $p=1$. Let us define $F'(x):=F'(x;x^k)$ to simplify the notation and note that $F'(x)$ takes the summation form $F'(x)=H(x)+G(x)$. Assume that $H(\cdot)$ is Lipschitz continuous with constant $L_H$ and strongly monotone with modulus $\mu_H$, and $G(\cdot)$ is monotone. Consider the approximation operator $\tilde F'(x;y):=H(y)+G(x)$, and we have:
\[
\left\|\tilde F'(x;y)-F(x)\right\|=\left\|H(x)-H(y)\right\|\le L_H\|x-y\|.
\]
Therefore, based on the results in Section \ref{sec:ARE-global} (in particular Theorem \ref{thm:are-global-strong}), the following update procedure
\begin{eqnarray}
\left\{
\begin{array}{ccl}
\bar x^{t+0.5} &:=& \VI_{\XX}\left(H(\bar x^t)+G(x)+L_H(x-\bar x^t)\right) , \\
&&\\
\bar x^{t+1} &:=& \mbox{arg}\min\limits_{x \in \XX} \, \langle H(\bar x^{t+0.5})+G(\bar x^{t+0.5}) , x - \bar x^t \rangle + \frac{L_H}{2} \|x-\bar x^t\|^2,
\end{array}
\right.\label{asy-eg-update}
\end{eqnarray}
for $t=0,1,2,...$ is guaranteed to converge to an $\bar \epsilon$ solution with iteration complexity $\mathcal{O}\left(\frac{L_H}{\mu_H}\log\frac{1}{\bar \epsilon}\right)$. Indeed, method \eqref{asy-eg-update} is nothing but an ARE update instance \eqref{are-update}, where $F(x):=F'(x)=H(x)+G(x)$ and $\tilde F(x;\bar x^t):=\tilde F'(x;\bar x^t)=H(\bar x^t)+G(x)$, and $p=1$. Let us denote $H(x)=F(x^k)+\nabla F(x^k)(x-x^k)$ and $G(x)=L_2\|x-x^k\|(x-x^k)$. Indeed, $H(x)$ is Lipschitz continuous with $L_H=\|\nabla F(x^k)\|\le L$, and $H(x)+G(x)$ is strongly monotone with $\mu_H=\mu>0$ provided the original operator $F(x)$ is strongly monotone with $\mu>0$. Under this formulation, solving subproblem (S) \eqref{subproblem-p-2} is equivalent to solving $\VI_{\XX}\left(F'(x)\right)=\VI_{\XX}\left(H(x)+G(x)\right)$ and can be solved approximately by the iterative procedure \eqref{asy-eg-update} with iteration complexity $\mathcal{O}\left(\frac{L}{\mu}\log\frac{1}{\bar \epsilon}\right)$.

We now show that each iteration of \eqref{asy-eg-update} can be further reduced to two gradient projection steps, whose computational cost is significantly reduced compared to directly solving the original VI subproblem (S). Note that the second step for updating $\bar x^{t+1}$ requires a gradient projection step, while the first step requires solving $\bar x^{t+0.5}$ from a VI problem in the following form:
\begin{eqnarray}
\langle H(\bar x^t)+L_2\|\bar x^{t+0.5}- x^k\|(\bar x^{t+0.5}- x^k)+L_H(\bar x^{t+0.5}-\bar x^t),x-\bar x^{t+0.5}\rangle\ge0,\quad\forall x\in\XX,\nonumber
\end{eqnarray}
which is optimality condition of the optimization problem:
\begin{eqnarray}
\min\limits_{x\in\XX}\quad \frac{L_2}{3}\|x- x^k\|^3+\frac{L_H}{2}\|x-\bar x^t\|^2+H(\bar x^t)^\top x.\label{cubic-opt-sub}
\end{eqnarray}
The following analysis adopts a similar reformulation as proposed in \cite{nesterov2006constrained} to solve \eqref{cubic-opt-sub} . Let us first reformulate \eqref{cubic-opt-sub} into
\begin{eqnarray}
&&\arg\min\limits_{x\in\XX}\quad \frac{L_2}{3}\|x- x^k\|^3+\frac{L_H}{2}\|x-\bar x^t\|^2+H(\bar x^t)^\top x\nonumber\\
&=& \arg\min\limits_{x\in\XX}\quad \frac{L_2}{3}\|x- x^k\|^3+\frac{L_H}{2}\left(\|x-x^k\|^2+\|x-\bar x^t\|^2-\|x-x^k\|^2\right)+H(\bar x^t)^\top x\nonumber\\
&=& \arg\min\limits_{x\in\XX}\quad \frac{L_2}{3}\|x- x^k\|^3+\frac{L_H}{2}\|x-x^k\|^2+L_H(x^k-\bar x^t)^\top x+H(\bar x^t)^\top x\nonumber\\
&=& \arg\min\limits_{x\in\XX}\quad \frac{L_2}{3}\|x- x^k\|^3+\frac{L_H}{2}\|x-x^k\|^2+\left(L_H(x^k-\bar x^t)+H(\bar x^t)\right)^\top (x-x^k)\nonumber.
\end{eqnarray}
Denote
\[
g_t(x^k)=L_H(x^k-\bar x^t)+H(\bar x^t)=F(x^k)+\left(\nabla F(x^k)-I\right)(\bar x^t-x^k),
\]
for a given fixed $x^k$. Since
\[
\frac{1}{3}r^3=\max\limits_{\tau\ge0}\quad r^2\tau-\frac{2}{3}\tau^{\frac{3}{2}},
\]
we have
\begin{eqnarray}
&&\min\limits_{x\in\XX}\quad \frac{L_2}{3}\|x- x^k\|^3+\frac{L_H}{2}\|x-x^k\|^2+g_t(x^k)^\top (x-x^k)\nonumber\\
&=& \min\limits_{x\in\XX}\max\limits_{\tau\ge0}\quad L_2\left(\tau\|x-x^k\|^2-\frac{2}{3}\tau^{\frac{3}{2}}\right)+\frac{L_H}{2}\|x-x^k\|^2+g_t(x^k)^\top (x-x^k)\nonumber\\
&=&\max\limits_{\tau\ge0}\left( -\frac{2}{3}L_2\tau^{\frac{3}{2}}+\min\limits_{x\in\XX}\left\{g_t(x^k)^\top(x-x^k)+\left(L_2\tau+\frac{L_H}{2}\right)\|x-x^k\|^2\right\}\right).\label{max-min-subprob}
\end{eqnarray}
The inner minimization gives a closed-form solution, denoted as
\[
\bar x^{t+0.5}(\tau) = \mathbf{P}_{\XX}\left(x^k-\frac{1}{2L_2\tau+L_H}g_t(x^k)\right),
\]
where $\mathbf{P}_{\XX}$ is the projection operator onto $\XX$. On the other hand, the outer maximization of \eqref{max-min-subprob} is a simple one-dimensional concave maximization, and the solution given the expression of $\bar x^{t+0.5}(\tau)$ is given by the following:
\[
\tau^*=\arg\max\limits_{\tau}\quad -\frac{2}{3}L_2\tau^{\frac{3}{2}}-\frac{1}{4L_2\tau+2L_H}\|g_t(x^k)\|^2,
\]
which can be solved efficiently by any common software.

To summarize, the update process in \eqref{asy-eg-update} can be rewritten into
\begin{eqnarray}
\mbox{(SS2)}\quad\left\{
\begin{array}{ll}
     & \tau^*:=\arg\max\limits_{\tau}\quad -\frac{2}{3}L_2\tau^{\frac{3}{2}}-\frac{1}{4L_2\tau+2L_H}\|g_t(x^k)\|^2, \\
     \\
     & \bar x^{t+0.5}=\bar x^{t+0.5}(\tau^*) := \mathbf{P}_{\XX}\left(x^k-\frac{1}{2L_2\tau^*+L_H}g_t(x^k)\right),\\
     \\
     & \bar x^{t+1} := \mbox{arg}\min\limits_{x \in \XX} \, \langle H(\bar x^{t+0.5})+G(\bar x^{t+0.5}) , x - \bar x^t \rangle + \frac{L_H}{2} \|x-\bar x^t\|^2,
\end{array}
\right.\label{subprob-SS2}
\end{eqnarray}
whose major cost lies in performing two gradient projection steps. Process \eqref{subprob-SS2} is then preformed iteratively until we obtain an approximate solution $\|\bar x^t-x^{k+0.5}\|\le\bar \epsilon$ with iteration complexity $\mathcal{O}\left(\frac{L}{\mu}\log\frac{1}{\bar \epsilon}\right)$. Compared to the methods discussed in Section \ref{sec:subproblem-1}, the method discussed in this section in general can require more inner iterations to operate with, but at the same time solving subproblem (SS2) can also be performed with much less cost than solving (SS1).

\section{Structured ARE Schemes}
\label{sec:structured-are}

In the previous sections, we analyze the convergence properties of ARE in its general update form \eqref{are-update} without specifying the approximation mapping $\tilde F(x;x^k)$ and the corresponding order $p$. Such general form is powerful in that it enables us to establish unified analysis for potentially many different specific methods. We devote this section to the discussion on some of these examples and the connections to existing methods in the literature. 

Consider the following structured VI where the operator is given in the {\it composite form}:
\begin{eqnarray}
F(x)=H(x)+G(x).\label{composite-opt}
\end{eqnarray}
We shall discuss different realizations of ARE in solving $\VI_\mathcal{X}\left(F(x)\right)$. The first immediate example is the extra-gradient method, which is equivalent to ARE with $p=1$ and $\tilde F(x;x^k):=F(x^k)$:
\begin{eqnarray}
\left\{
\begin{array}{ccl}
x^{k+0.5} &:=& \VI_{\XX} \left(  F(x^k) + L_1  (x-x^k) \right) , \\
x^{k+1} &:=& \mbox{arg}\min\limits_{x \in \XX} \, \langle F(x^{k+0.5}) , x - x^k \rangle + \frac{L_1}{2} \|x-x^k\|^2.
\end{array}
\right.\label{ARE-eg}
\end{eqnarray}
The extra-gradient method \eqref{ARE-eg} treats $F(x)$ as a single operator without using the specific composite structure \eqref{composite-opt}. The proposed ARE, however, provides the possibilities of using alternative approximation operator $\tilde F(x;x^k)$ in the update. In particular, consider the case $p=1$ and $\tilde F(x;x^k):=H(x^k)+G(x)$:
\begin{eqnarray}
\left\{
\begin{array}{ccl}
x^{k+0.5} &:=& \VI_{\XX} \left(  H(x^k)+G(x) + L_1  (x-x^k) \right) , \\
x^{k+1} &:=& \mbox{arg}\min\limits_{x \in \XX} \, \langle F(x^{k+0.5}) , x - x^k \rangle + \frac{L_1}{2} \|x-x^k\|^2,
\end{array}
\right.\label{ARE-fw-1}
\end{eqnarray}
where the update of $x^{k+0.5}$ can be viewed as a combined gradient-projection/proximal point step, or a {\it proximal gradient} (projection) update. As we have also seen in Section \ref{sec:subproblem-2}, the potential advantage of performing such update is that we are able to relax the Lipschitz continuity assumption made for the overall operator $F(x)$. Indeed, the required condition for the iteration complexity results to hold for ARE is given by \eqref{approx-oper-bound}, while in the update \eqref{ARE-fw-1} we have:
\begin{eqnarray}
\|\Tilde{F}(x;x^k)-F(x)\|=\|H(x^k)-H(x)\|\le L_H\|x-x^k\|,\nonumber
\end{eqnarray}
where we assume $L_H$ is the Lipschitz constant for $H(x)$. Therefore, only the Lipschitz conitnuity of $H(x)$ is required, and the constant $L_1$ in \eqref{ARE-fw-1} can be replaced with $L_H$. In the example discussed in Section \ref{sec:subproblem-2}, the VI subproblem for solving $x^{k+0.5}$ can be further reduced to simpler forms if $G(x)$ is the gradient of some convex function.

We can also extend the original problem $\VI_\XX(F(x))$ to the monotone inclusion problem:
\begin{eqnarray}
0\in H(x^*)+G(x^*),\nonumber
\end{eqnarray}
where $G:\XX\rightrightarrows \mathbb{R}^n$ is a set-valued maximal monotone operator. Well-known maximal monotone operators include $\partial f$, the subdifferential of a proper closed convex function $f$, and $N_{\mathcal{X}}(x)$, the normal cone of a closed convex set $\mathcal{X}$. For further discussion regarding maximal monotone operator and monotone inclusion problem, the interested readers are referred to \cite{facchinei2007finite}. In view of the monotone inclusion problem, the same ARE update discussed earlier \eqref{ARE-fw-1} can be adjusted accordingly:
\begin{eqnarray}
\left\{
\begin{array}{ccl}
x^{k+0.5} &:=& \VI_{\XX} \left(  H(x^k)+u^{k+0.5} + L_1  (x-x^k) \right) , \\
x^{k+1} &:=& \mbox{arg}\min\limits_{x \in \XX} \, \langle H(x^{k+0.5})+u^{k+0.5} , x - x^k \rangle + \frac{L_1}{2} \|x-x^k\|^2,
\end{array}
\right.\label{ARE-fw-2}
\end{eqnarray}
where $u^{k+0.5}\in G(x^{k+0.5})$. Note that in the above expression, $x^{k+0.5}$ equivalently satisfies the following equation:
\begin{eqnarray}
H(x^k)+u^{k+0.5}+L_1(x^{k+0.5}-x^k)=0.\nonumber
\end{eqnarray}
Substituting $u^{k+0.5}=-H(x^k)-L_1(x^{k+0.5}-x^k)$ in the update of $x^{k+1}$, we get the following scheme:
\begin{eqnarray}
\left\{
\begin{array}{ccl}
x^{k+0.5} &:=& \VI_{\XX} \left(  H(x^k)+u^{k+0.5} + L_1  (x-x^k) \right) , \\
x^{k+1} &:=& \mbox{arg}\min\limits_{x \in \XX} \, \langle H(x^{k+0.5})-H(x^k) , x - x^k \rangle + \frac{L_1}{2} \|x-x^{k+0.5}\|^2,
\end{array}
\right.\label{ARE-fw-mod}
\end{eqnarray}
which is the modified forward-backward update proposed in \cite{tseng2000modified} by noticing that $x^{k+0.5}$ in \eqref{ARE-fw-mod} is also equivalent to the forward-backward step $x^{k+0.5}=\left(I+\frac{1}{L_1}G\right)^{-1}\left(I-\frac{1}{L_1}H\right)(x^k)$. This shows that the modified forward-backward method is indeed another important instance of ARE for the more general monotone inclusion problem.

Moving forward to the higher-order $(p\ge 2)$ ARE schemes, an immediate example is to take the Taylor approximation $\tilde F(x;x^k):=\sum\limits_{i=0}^{p-1}\frac{1}{i!}\nabla^{i}F(x^k)[x-x^k]^i$, resulting in the following update:
\begin{eqnarray}
\left\{
\begin{array}{ccl}
x^{k+0.5} &:=& \VI_{\XX} \left(  \sum\limits_{i=0}^{p-1}\frac{1}{i!}\nabla^{i}F(x^k)[x-x^k]^i + L_p \|x-x^k\|^{p-1} (x-x^k) \right) , \\
x^{k+1} &:=& \mbox{arg}\min\limits_{x \in \XX} \, \langle F(x^{k+0.5}) , x - x^k \rangle + \frac{L_p\|x^{k+0.5}-x^k\|^{p-1}}{2} \|x-x^k\|^2.
\end{array}
\right.\label{ARE-homp}
\end{eqnarray}
The above update \eqref{ARE-homp} can be viewed as equivalent forms of the NPE \cite{monteiro2012iteration} ($p=2$) and higher-order mirror-prox \cite{bullins2020higher, adil2022optimal} $(p\ge 2)$. In view of the previous discussion, it is then natural to consider the higher-order approximation operator in the form $\tilde F(x;x^k):=\sum\limits_{i=0}^{p-1}\frac{1}{i!}\nabla^{i}H(x^k)[x-x^k]^i+G(x)$ for the specific composite structure \eqref{composite-opt}, and the next scheme follows:
\begin{eqnarray}
\left\{
\begin{array}{ccl}
x^{k+0.5} &:=& \VI_{\XX} \left(  \sum\limits_{i=0}^{p-1}\frac{1}{i!}\nabla^{i}H(x^k)[x-x^k]^i+G(x) + L_p \|x-x^k\|^{p-1} (x-x^k) \right) , \\
x^{k+1} &:=& \mbox{arg}\min\limits_{x \in \XX} \, \langle F(x^{k+0.5}) , x - x^k \rangle + \frac{L_p\|x^{k+0.5}-x^k\|^{p-1}}{2} \|x-x^k\|^2.
\end{array}
\right.\label{ARE-homp-2}
\end{eqnarray}
The above scheme \eqref{ARE-homp-2} can be viewed as generalization of several existing methods. In addition to generalizing the higher-order mirror-prox method, it also generalizes the modified forward-backward method (in the form \eqref{ARE-fw-1}) to $p^\T$-order. Furthermore, it generalizes the tensor method proposed in \cite{doikov2019local} for composite optimization to solving composite VI. Indeed, consider the following problem:
\begin{eqnarray}
\min\limits_{x\in\XX}\quad h(x)+g(x).\nonumber
\end{eqnarray}
To simplify the discussion, assume both $h,g$ are convex and differentiable and denote $H(x):=\nabla h(x)$ and $G(x):=\nabla g(x)$. The following inequality defines the solution $x^{k+0.5}$ in \eqref{ARE-homp-2}:
\begin{eqnarray}
\langle \sum\limits_{i=0}^{p-1}\frac{1}{i!}\nabla^{i}H(x^k)[x^{k+0.5}-x^k]^i+G(x^{k+0.5}) + L_p \|x^{k+0.5}-x^k\|^{p-1} (x^{k+0.5}-x^k),x-x^{k+0.5}\rangle\ge0,\quad\forall x\in\XX,\nonumber
\end{eqnarray}
which is equivalent to
\begin{eqnarray}
x^{k+0.5}:=\arg\min\limits_{x\in\XX}\sum\limits_{i=1}^{p}\frac{1}{i!}\nabla^{i}h(x^k)[x-x^k]^i+g(x)+\frac{L_p}{p+1}\|x-x^k\|^{p+1}.\label{tensor-composite-VI}
\end{eqnarray}
Note that in the context of composite optimization \cite{doikov2019local}, the problem is unconstrained and the minimization step \eqref{tensor-composite-VI} is performed over the domain of $g(x)$. In addition, while the acceleration in optimization requires an additional sequence $\{y^k\}$ so that the update \eqref{tensor-composite-VI} is performed at $y^k$ instead of $x^k$ (for example, FISTA \cite{beck2009fast} and Nesterov's accelerated tensor method \cite{nesterov2018implementable}), the acceleration in VI in general takes the form of the extra-gradient step such as the update of $x^{k+1}$ in \eqref{ARE-homp-2}.

Next, we further consider the following more general composite VI model with the operator:
\begin{eqnarray}
F(x)=H(G(x)).\label{general-composite-opt}
\end{eqnarray}
Obviously, if we let $G(x)=G_1(x)+G_2(x)$ and $H(x)=x$, the general composite operator \eqref{general-composite-opt} reduces to the special case in the summation form \eqref{composite-opt}. This general model enables us to extend the approximation schemes discussed earlier, as shown in the next two examples. The first example is an {\it outer approximation}:
\begin{eqnarray}
\tilde{F}(x;x^k):=\tilde H(G(x);G(x^k)),\nonumber
\end{eqnarray}
which replaces the outer operator $H(\cdot)$ with an approximation operator $\tilde H(\cdot\,;y)$ that satisfies the condition \eqref{approx-oper-bound} with some fixed $y$ and constant $L_p:=L_H$. The resulting overall approximation $\tilde F(x;x^k)$ hence satisfies \eqref{approx-oper-bound} as well from the following bound:
\begin{eqnarray}
\left\|\tilde F(x;x^k)-F(x)\right\|&=&\left\|\tilde H(G(x);G(x^k))-H(G(x))\right\|\nonumber\\
&\le& \tau L_H\left\|G(x)-G(x^k)\right\|^p\le \tau L_HL_G^p\|x-x^k\|^p,\nonumber
\end{eqnarray}
where we also assume $G(x)$ is Lipschitz continuous with constant $L_G$. As an exemplifying scheme, let $\tilde H(\cdot\,;y)$ be the Taylor approximation of $H(\cdot)$ with $p=2$, which results in the next update scheme:
\begin{eqnarray}
\left\{
\begin{array}{ccl}
x^{k+0.5} &:=& \VI_{\XX} \left(  H(G(x^k))+\nabla H(G(x^k))\left(G(x)-G(x^k)\right) + L_HL_G^2 \|x-x^k\| (x-x^k) \right) , \\
x^{k+1} &:=& \mbox{arg}\min\limits_{x \in \XX} \, \langle F(x^{k+0.5}) , x - x^k \rangle + \frac{L_HL_G^2\|x^{k+0.5}-x^k\|}{2} \|x-x^k\|^2.
\end{array}
\right.\label{outer-linearization-update}
\end{eqnarray}
Note that in this example, even if the outer approximation operator $\tilde H(\cdot\,;y)$ is the Taylor approximation, the overall approximation operator $\tilde F(x;x^k):=H(G(x^k))+\nabla H(G(x^k))\left(G(x)-G(x^k)\right)$ is not (it is not even linear unless $G(\cdot)$ is). In general, $\tilde H(\cdot\,;y)$ needs not be the Taylor approximation but can be any approximation satisfying \eqref{approx-oper-bound}, such as the ones discussed earlier.

The second example based on the composite VI model \eqref{general-composite-opt} is an {\it inner approximation}:
\begin{eqnarray}
\tilde F(x;x^k):=H(\tilde{G}(x;x^k)),\nonumber
\end{eqnarray}
which replaces the inner operator $G(x)$ with an approximation operator $\tilde G(x;x^k)$ that satisfies the condition \eqref{approx-oper-bound}, with the constant now defined as $L_p:=L_G$. Similarly, we have:
\begin{eqnarray}
\left\|\tilde F(x;x^k)-F(x)\right\|&=&\left\| H(\tilde G(x;x^k))-H(G(x))\right\|\nonumber\\
&\le& L_H\left\|\tilde G(x;x^k)-G(x)\right\|\le \tau L_HL_G\|x-x^k\|^p,\nonumber
\end{eqnarray}
which indicates that $\tilde F(x;x^k)$ also satisfies \eqref{approx-oper-bound} as long as $H(\cdot)$ is Lipschitz continuous with $L_H$. If $\tilde G(x;x^k)$ is the Taylor approximation of $G(x)$ at $x^k$ with $p=2$, we have the following scheme:
\begin{eqnarray}
\left\{
\begin{array}{ccl}
x^{k+0.5} &:=& \VI_{\XX} \left(  H\left(G(x^k)+\nabla G(x^k)(x-x^k)\right) + L_HL_G \|x-x^k\| (x-x^k) \right) , \\
x^{k+1} &:=& \mbox{arg}\min\limits_{x \in \XX} \, \langle F(x^{k+0.5}) , x - x^k \rangle + \frac{L_HL_G\|x^{k+0.5}-x^k\|}{2} \|x-x^k\|^2.
\end{array}
\right.\label{inner-linearization-update}
\end{eqnarray}
Again, the overall approximation operator $\tilde F(x;x^k)$ is not Taylor approximation even if the inner approximation operator $\tilde G(x;x^k)$ is, and it is not linear unless $H(\cdot)$ is. While the examples in \eqref{outer-linearization-update} and \eqref{inner-linearization-update} use a similar concept to construct the approximation operator $\tilde F(x;x^k)$, the resulting update scheme can be quite different given how we identify the specific composite structure $F(x)=H(G(x))$ in a problem. 

In this section, we first discuss several structured ARE schemes based on the composite form of the operator \eqref{composite-opt}, which either coincides with or generalizes existing methods. We further discuss the more general composite VI model \eqref{general-composite-opt} and present two different examples to illustrate the concept of outer approximation and inner approximation. We remark that composite VI may take even more general forms such as multiple layers $F(x)=H_1(H_2(...(H_n(x))))$, or multiple blocks $F(x)=H(G_1(x),G_2(x),...,G_n(x))$, or arbitrary combinations of these two. Developing specific schemes based on these composite forms can be highly dependent on each individual problem at hand and the subproblem of solving $x^{k+0.5}$ may be difficult. However, the purpose of this paper is to reveal the potentials of a general scheme %concept 
of approximation used in the ARE framework, by pointing out possibilities other than the most commonly applied Taylor approximations in many existing schemes.  By taking the structure of the VI operator into consideration, ARE can possibly include even more complicated schemes than the ones discussed in this section.  
%that are also new to the literature. 
As long as certain assumptions are satisfied and one is able to develop efficient subroutines for solving the VI subproblem, the results established in earlier sections can immediately provide optimal iteration complexity guarantee for the new scheme.

\section{Numerical Experiments}
\label{sec:numerical}
In this section, we examine the convergence of ARE and ARE-restart with $p=2$ and compare the performance with other common first-order methods. We consider the following unconstrained saddle point problem in the experiment:
\begin{eqnarray}
    \min\limits_{x\in\mathbb{R}^n}\max\limits_{y\in\mathbb{R}^m}f(x,y) & = & \frac{1}{M_1}\sum\limits_{i=1}^{M_1}\ln(1+e^{-a_i^{\top}x})+\frac{\lambda}{2}\|x\|^2 \nonumber\\
    && +x^\top Ay-\frac{1}{M_2}\sum\limits_{j=1}^{M_2}\ln(1+e^{-b_j^{\top}y})-\frac{\lambda}{2}\|y\|^2.\label{saddle-point-prob}
\end{eqnarray}
To transform the saddle point problem \eqref{saddle-point-prob} into equivalent VI formulation, let us redefine the VI variable as $u=(x,y)^\top$ and the operator
\begin{eqnarray}
    F(u) = \begin{pmatrix}-\frac{1}{M_1}\sum\limits_{i=1}^{M_1}\frac{a_i}{1+e^{-a_i^\top x}}+\lambda x+Ay\\-\frac{1}{M_2}\sum\limits_{j=1}^{M_2}\frac{b_j}{1+e^{-b_j^\top y}}+\lambda y-A^\top x\end{pmatrix},\nonumber
\end{eqnarray}
with the problem defined as $F(u)=0$. The ARE and ARE-restart implemented in the experiment specifically use the Taylor approximation as the approximation operator $\tilde{F}(u,u^{k}):=F(u^k)+\nabla F(u^k)(u-u^k)$ and can be expressed as:
\begin{eqnarray}
\left\{
\begin{array}{ccl}
u^{k+0.5} &:=& \VI_{\XX} \left(  F(u^k)+\nabla F(u^k)(u-u^k) + L_2 \|u-u^k\| (u-u^k) \right) , \\
u^{k+1} &:=& \mbox{arg}\min\limits_{u \in \XX} \, \langle F(u^{k+0.5}) , u - u^k \rangle + \frac{L_2\|u^{k+0.5}-u^k\|}{2} \|u-u^k\|^2.
\end{array}
\right.\label{ARE-p2-Taylor-update}
\end{eqnarray}
Since the original saddle point problem is unconstrained, we have $\XX:=\mathbb{R}^n\times\mathbb{R}^m$, and the VI subproblem for solving $u^{k+0.5}$ is equivalent to solving the equation:
\begin{eqnarray}
F(u^k)+\nabla F(u^k)(u^{k+0.5}-u^k) + L_2 \|u^{k+0.5}-u^k\| (u^{k+0.5}-u^k)=0,\nonumber
\end{eqnarray}
which can be solved via a Newton method (see discussions in Section \ref{sec:subproblem-1}). For a more detailed implementation, the interested readers are referred to Section 4 in \cite{huang2022cubic}. The restart procedure of update \eqref{ARE-p2-Taylor-update} is described in Section \ref{sec:ARE-global-strong}, and we use a pre-defined number for inner iterations between each restart.

The experiment is conducted under Matlab 2018 environment, and the problem parameters are as follows. The number of date points is $M_1=M_2=100$; the problem dimensions are $m=2n=50$; the elements of $a_i,b_j,A$ are generated by independent standard normal distribution; the second-order smoothness constant $L_2$ is estimated as 0.3. Note that the operator $F(u)$ is strongly monotone with modulus $\lambda$, which is varied to observe different convergence behaviors. The purpose of the experiments is to verify the convergence of ARE and ARE-restart, and we use the first-order methods, extra-gradient and OGDA, as the benchmarks for comparison. The convergence is measured as $\|F(u)\|$, and the results are presented in Figure \ref{fig:lambda1}-\ref{fig:lambda0001}. 

\begin{figure}[htbp]
\centering
\begin{minipage}[t]{0.48\textwidth}
\centering
\includegraphics[width=7cm]{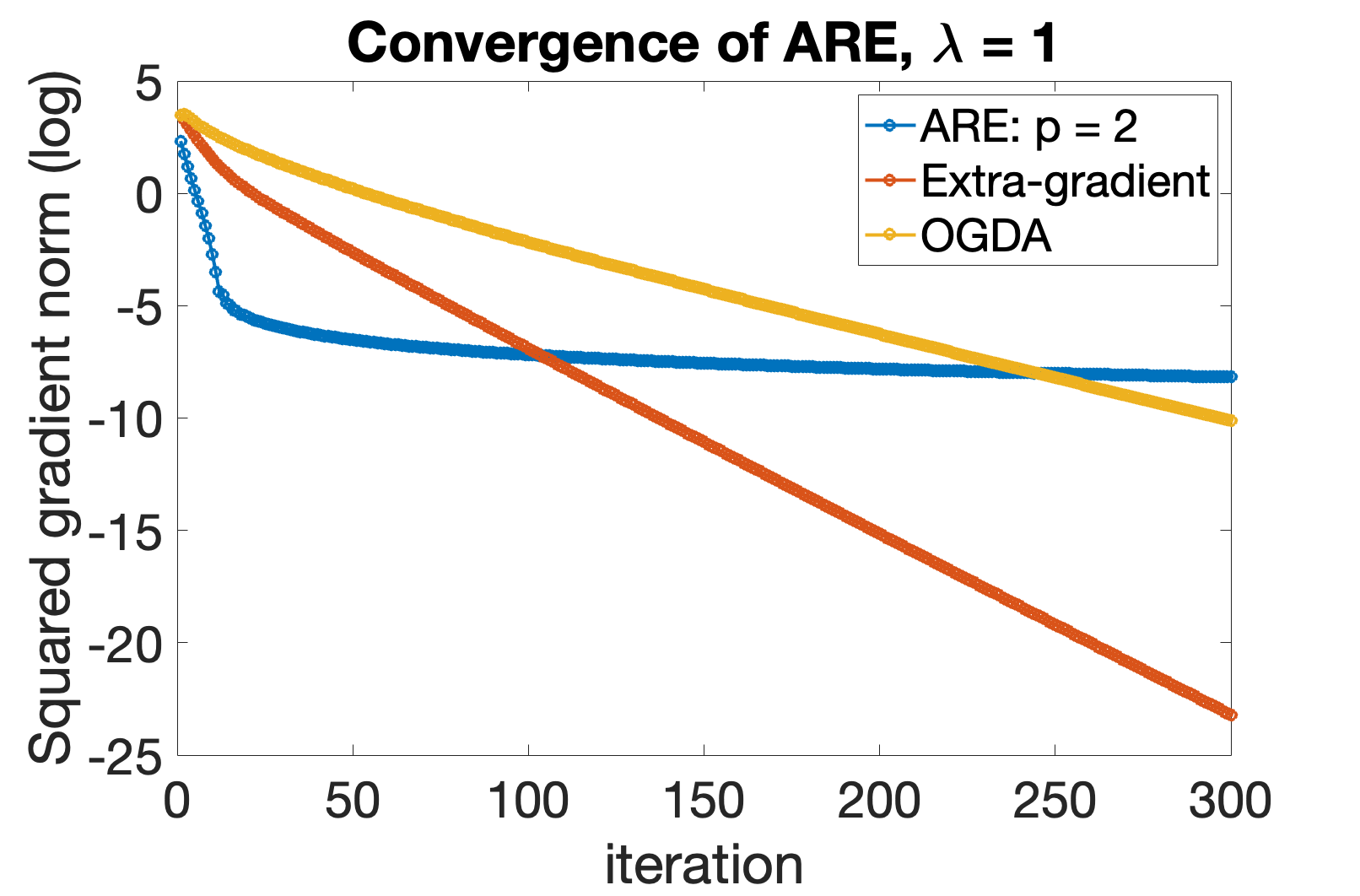}
\end{minipage}
\begin{minipage}[t]{0.48\textwidth}
\centering
\includegraphics[width=7cm]{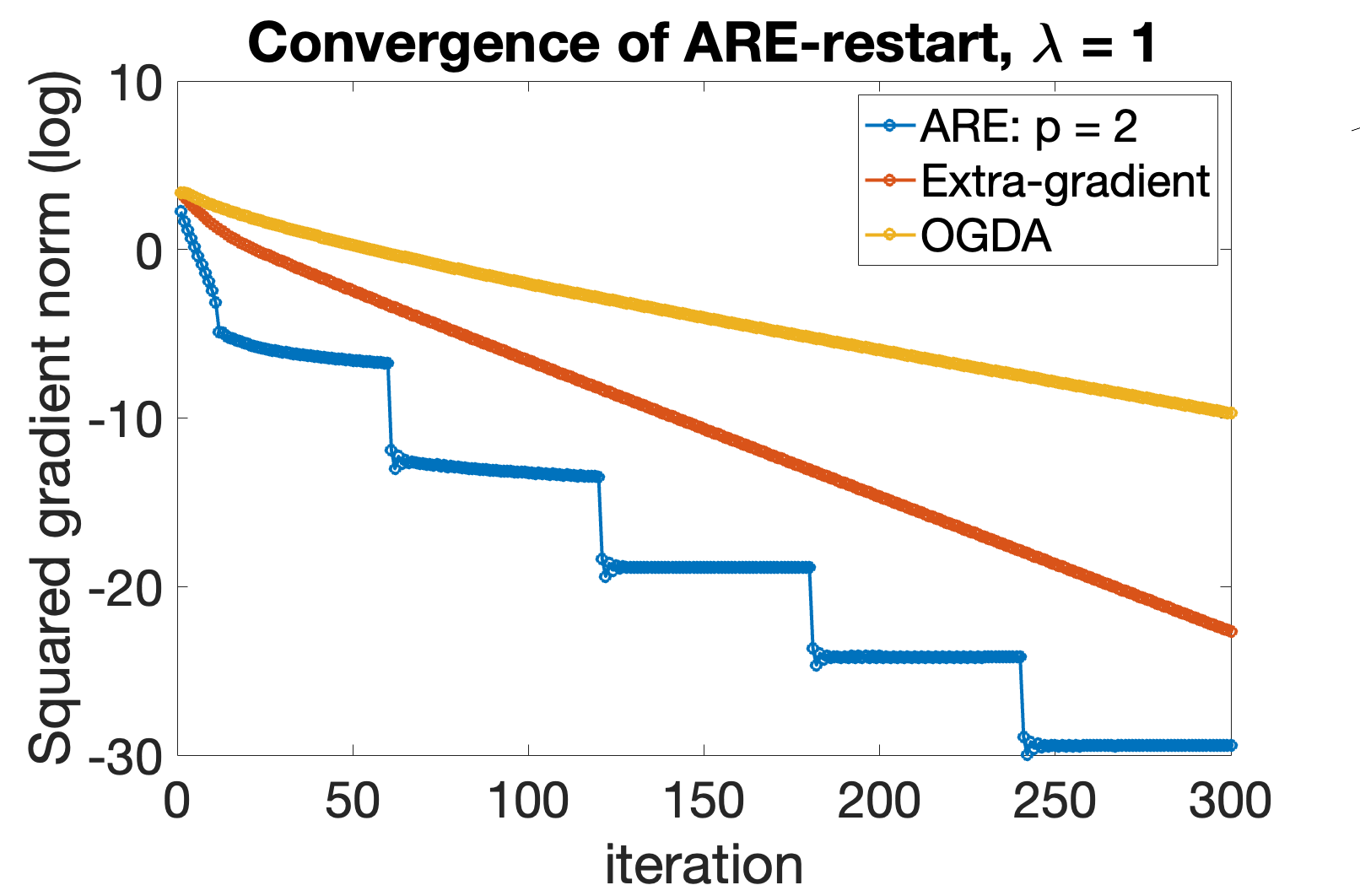}
\end{minipage}
\caption{Convergence in strongly monotone VI with $\lambda=1$}
\label{fig:lambda1}
\end{figure}
\begin{figure}[htbp]
\centering
\begin{minipage}[t]{0.48\textwidth}
\centering
\includegraphics[width=7cm]{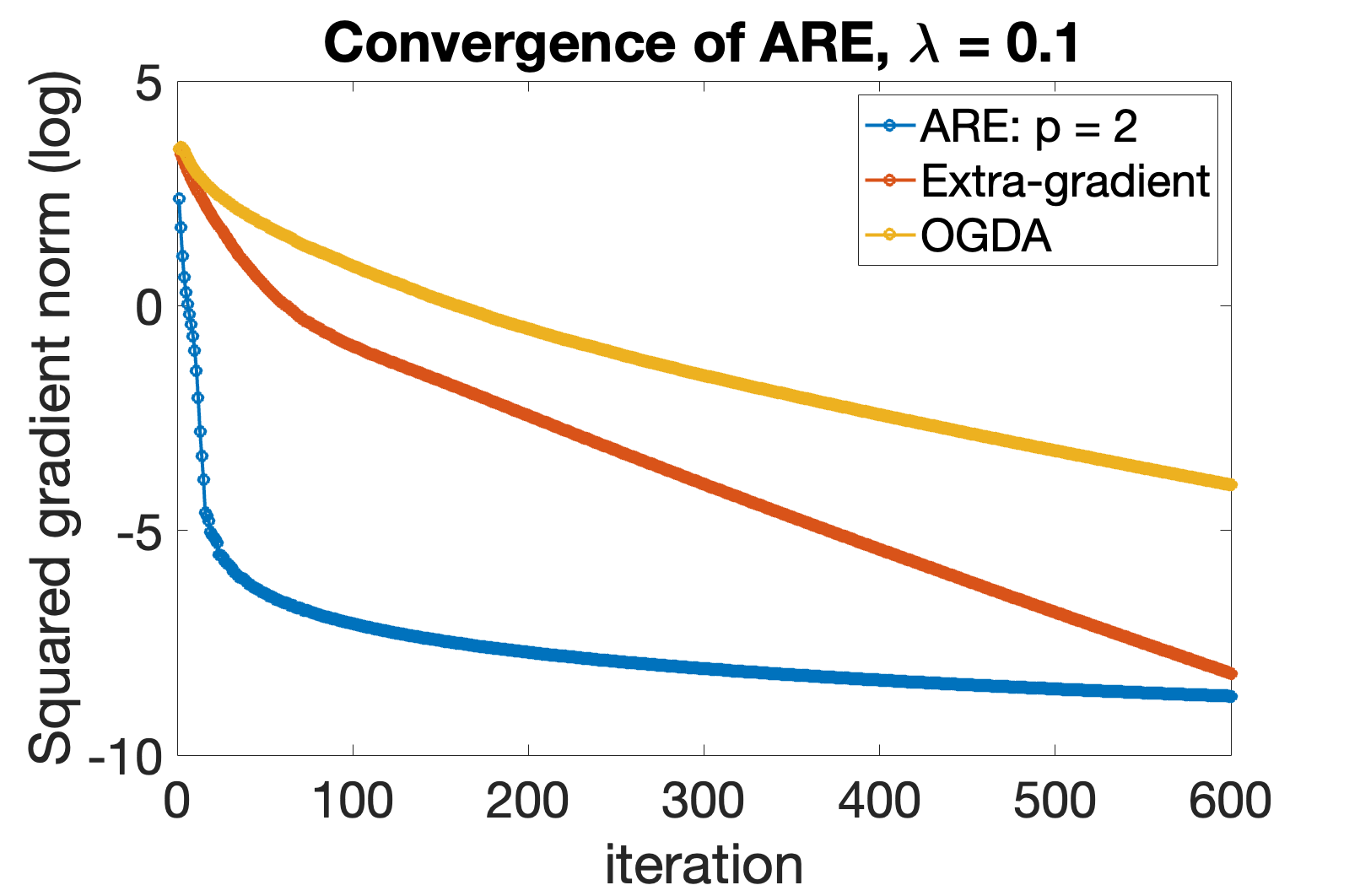}
\end{minipage}
\begin{minipage}[t]{0.48\textwidth}
\centering
\includegraphics[width=7cm]{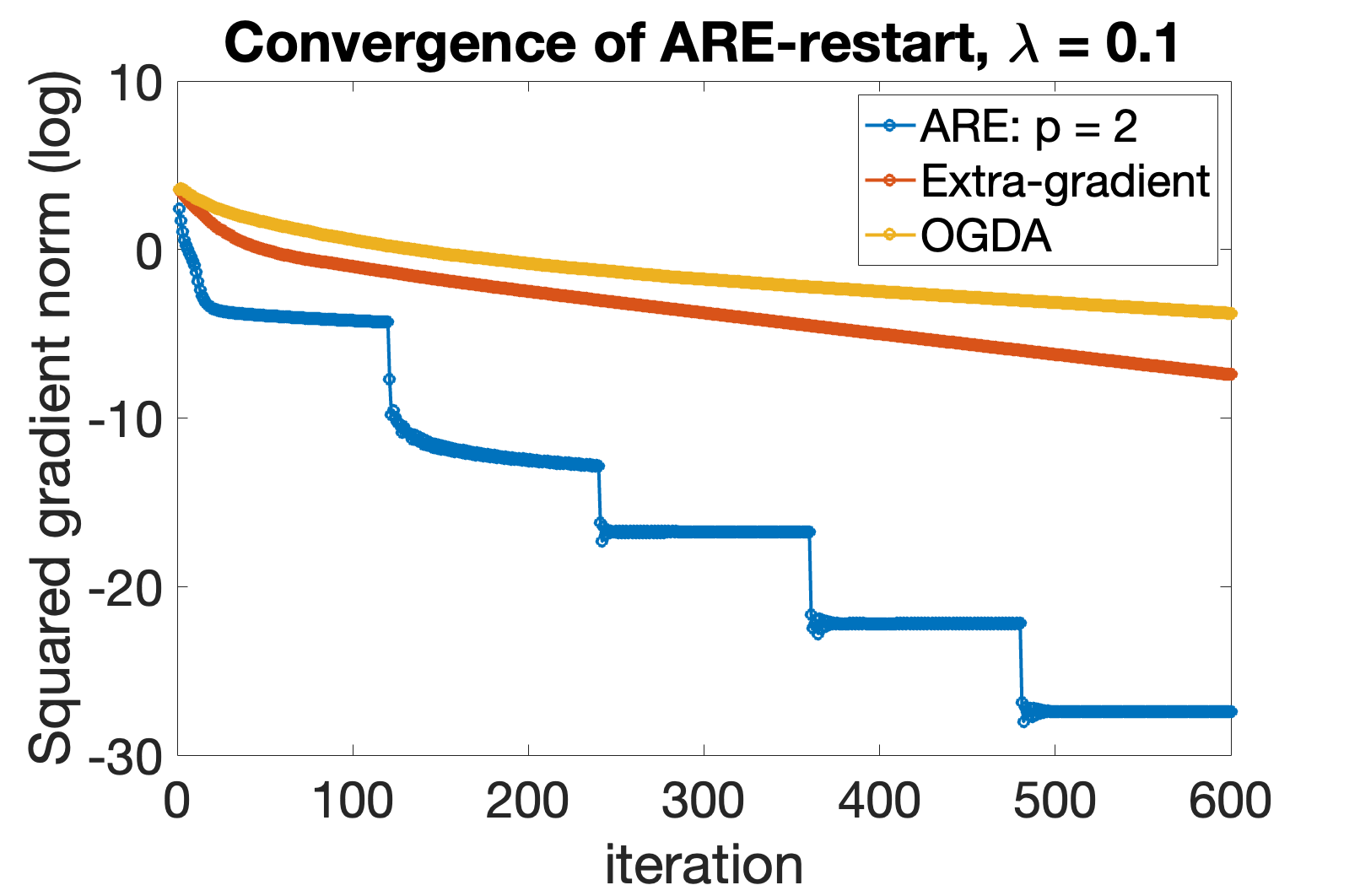}
\end{minipage}
\caption{Convergence in strongly monotone VI with $\lambda=0.1$}
\label{fig:lambda01}
\end{figure}

The convergence of ARE are shown in the left plots of Figure \ref{fig:lambda1}, Figure \ref{fig:lambda01}, and Figure \ref{fig:lambda0001}. All of them show clear sublinear convergence for the averaged iterates
\[
\Bar{u}_k:=\frac{\sum\limits_{i=1}^k\frac{u^{i+0.5}}{\gamma_i}}{\Gamma_k},\quad \Gamma_k:=\sum\limits_{i=1}^k\gamma_i^{-1},
\]
where $\gamma_i=L_2\|u^{i+0.5}-u^i\|$. Indeed, a sublinear convergence rate $\mathcal{O}\left(1/k^{\frac{3}{2}}\right)$ is guaranteed for ARE with $p=2$. However, when the problem is strongly monotone ($\lambda=1$), it will take significantly more iterations to converge to very high precision ($\|F(u)\|< 10^{-10}$) compared to the first-order methods extra-gradient and OGDA, which are designed to better exploit the strong monotonicity and admit linear convergence. However, when $\lambda$ is small ($\lambda=0.001$) and the problem becomes closer to a VI that is merely monotone, the performance of these first-order methods deteriorate fast to sublinear convergence that is significantly slower than ARE, a second-order method. On the other hand, ARE-restart (right plots of Figure \ref{fig:lambda1}, Figure \ref{fig:lambda01}, and Figure \ref{fig:lambda0001}) shows clear improvement over the first-order methods regardless of the strong monotonicity modulus $\lambda$. The process of restart is crucial in these experiments to take advantage of the strong monotonicity in the problem and bring the convergence of ARE beyond sublinear convergence to linear, or even superlinear, convergence. In the results shown in Figure \ref{fig:lambda1}-Figure \ref{fig:lambda0001}, the superlinear convergence happens immediately after each restart, followed by sublinear convergence in the rest of the epoch before the next restart. This particular convergence behavior enables the iterates of ARE-restart to quickly converge to high precision within much fewer iterations compared to ARE or other first-order methods.

\begin{figure}[htbp]
\centering
\begin{minipage}[t]{0.48\textwidth}
\centering
\includegraphics[width=7cm]{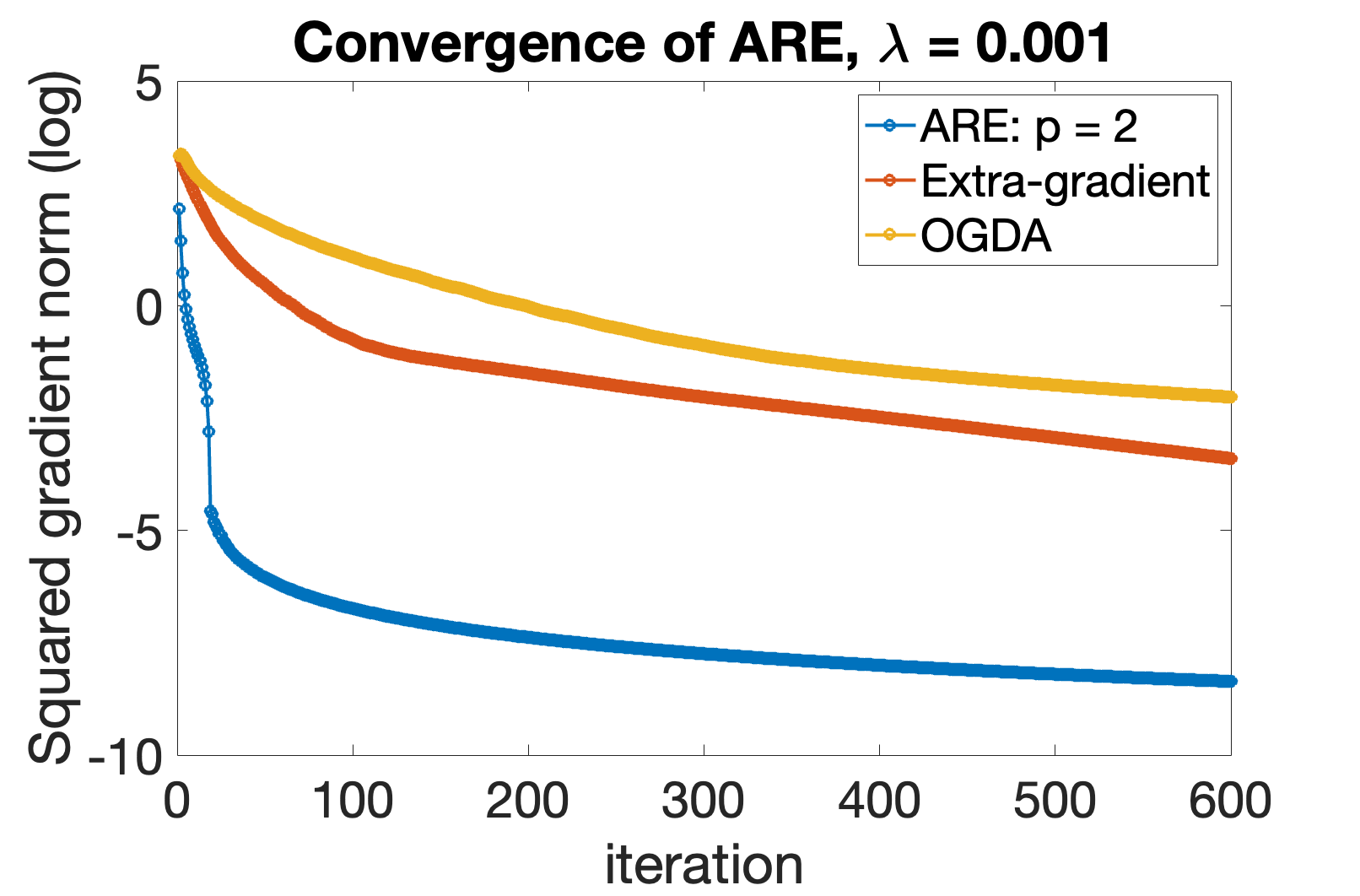}
\end{minipage}
\begin{minipage}[t]{0.48\textwidth}
\centering
\includegraphics[width=7cm]{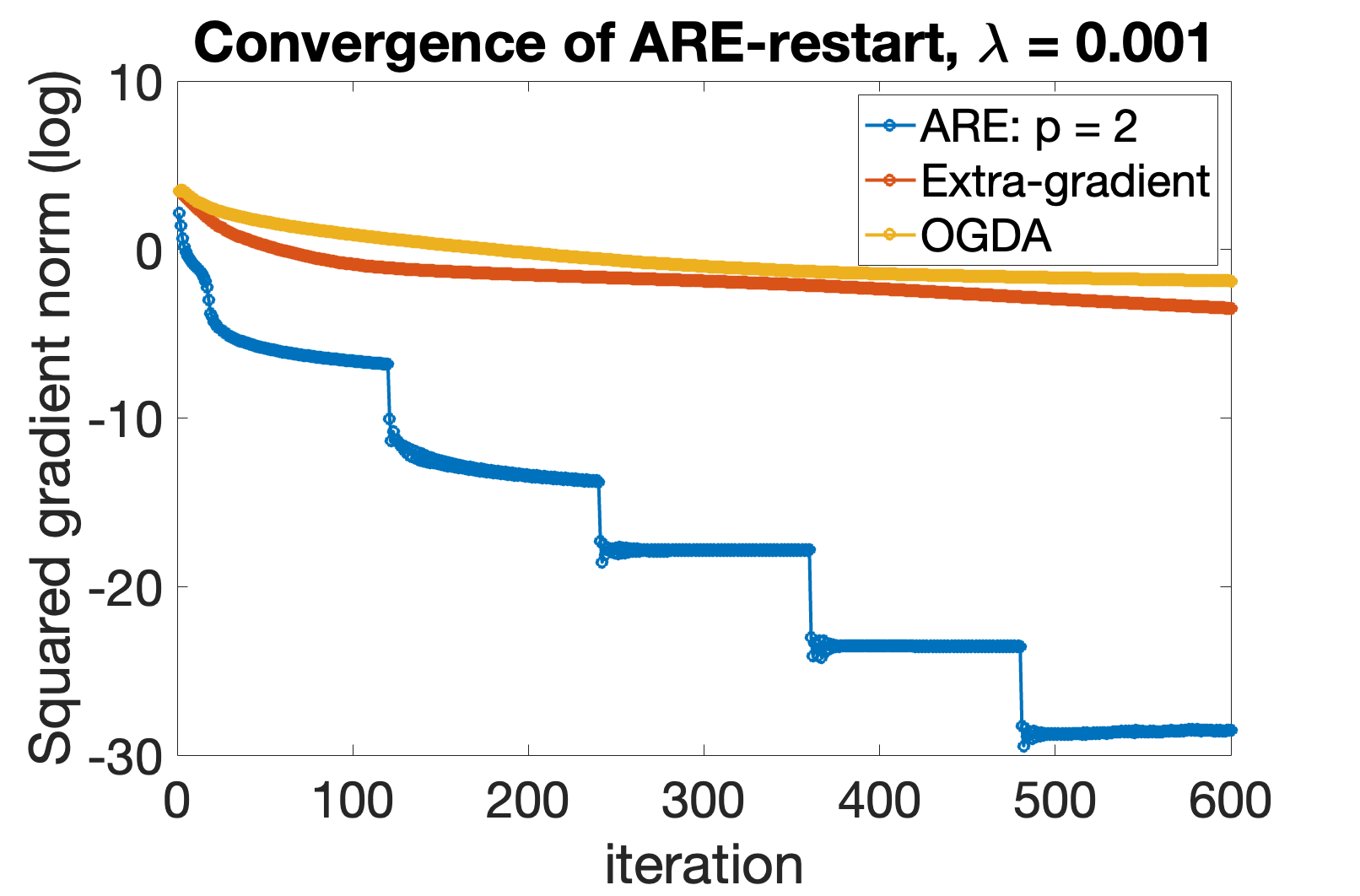}
\end{minipage}
\caption{Convergence in strongly monotone VI with $\lambda=0.001$}
\label{fig:lambda0001}
\end{figure}

\section{Conclusion}
\label{sec:conclusion}

In this paper, we propose the approximation-based regularized extra-gradient (ARE) scheme for solving monotone VI. The key feature of ARE is to solve a regularized VI subproblem in the first step, where the operator consists of a general approximation mapping satisfying a $p^\T$-order Lipschitz bound \eqref{approx-oper-bound} and the gradient mapping of a $(p+1)^\T$-order regularization. Iteration complexities are established for both monotone VI (ARE) and strongly monotone VI (ARE-restart), and the results match the lower bound for general $p^\T$-order methods. We further analyze the local convergence behavior for strongly monotone VI when $p>1$ and establish $p^\T$-order superlinear convergence, which is an improvement over the existing results. 

By introducing the general approximation mapping that satisfies the Lipschitz bound, ARE can be viewed as a more general framework that includes multiple existing methods in the literature. As a result, unified results can be established for different methods under the general ARE framework. We then discuss detailed implementations for solving the regularized VI subproblem under special cases, as well as some specialized ARE schemes if the VI operator has a composite structure. 

\printbibliography
% \bibliographystyle{acm}
% \bibliography{plan}

\end{document}